\documentclass[a4paper]{amsart}
\usepackage{amsthm}
\usepackage{amssymb}
\usepackage{amsmath}
\usepackage{graphicx,wrapfig}
\newtheorem{theorem}{Theorem}[section]
\newtheorem{lemma}[theorem]{Lemma}
\newtheorem{proposition}[theorem]{Proposition}
\newtheorem{remark}[theorem]{Remark}
\newtheorem{example}[theorem]{Example}

\newtheorem*{corollary*}{Corollary}
\newtheorem*{notation*}{Notation}
\newtheorem*{conjecture*}{Conjecture}

\numberwithin{equation}{section}

\gdef\myletter{}

\let\savetheequation\theequation
\def\theequation{\savetheequation\myletter}

\def\bv{\mathbf{v}}

\newcommand{\CC}{{\mathbb C}}

\newcommand{\RR}{{\mathbb R}}

\newcommand{\PP}{{\mathbb P}}

\renewcommand{\Im}{\mbox{Im}}

\renewcommand{\Re}{\mbox{Re}}

\def \bar{\overline}
\def \hat{\widehat}

\def \bv{{\bf v}}

\def \b0{{\bf 0}}

\usepackage{color}

\begin{document}

\title{Extremal functions for real convex bodies}

\author{D. Burns, N. Levenberg and S. Ma`u}


\address{University of Michigan, Ann Arbor, MI 48109-1043 USA}
\email{dburns@umich.edu}

\address{Indiana University, Bloomington, IN 47405 USA}

\email{nlevenbe@indiana.edu}

\address{University of Auckland, Private Bag 92019, Auckland, New Zealand}
\email{s.mau@auckland.ac.nz}
\date{\today}

\begin{abstract}We study the smoothness of the Siciak-Zaharjuta extremal function associated to a convex body in $\mathbb{R}^2$.  We also prove a formula relating the complex equilibrium measure of a convex body in $\mathbb{R}^n$ ($n\geq 2$) to that of its Robin indicatrix.  The main tool we use are extremal ellipses.\end{abstract}

\maketitle

\section{Introduction}

The Siciak-Zaharjuta extremal function for a compact set $K\subset\CC^n$ is the plurisubharmonic (psh) function on $\CC^n$ given by 
$$
V_K(z) := \sup\{u(z): u\in L(\CC^n), u\leq 0 \hbox{ on } K\} 
$$
where $L(\CC^n) = \{u \hbox{ psh on }\CC^n: \exists C\in\RR \hbox{ such that } u(z)\leq\log^+|z|+C \}$  denotes the class of psh functions on $\CC^n$ with logarithmic growth.

The uppersemicontinuous regularization  $V_K^*(z):=\limsup_{\zeta\to z}V_K(\zeta)$ is identically $+\infty$ if $K$ is pluripolar;  otherwise $V_K^*\in L(\CC^n)$; in fact, $V_K^*\in L^+(\CC^n)$ where $$L^+(\CC^n)=\{u\in L(\CC^n):\exists C\in\RR \hbox{ such that }u(z)\geq\log^+|z|+C\}.$$    The set $K$ is \emph{$L$-regular} if $K$ is non-pluripolar and $V_K=V_K^*$; this is equivalent to $V_K$ being continuous.  In this paper, $V_K$ will always have a continuous foliation structure that automatically gives $L$-regularity.

The complex Monge-Amp\`ere operator applied to a function $u$ of class $C^2$ on some domain in $\CC^2$ is given by
$$
(dd^c u)^n = i\partial\bar\partial u\wedge\cdots i\partial\bar\partial  u \quad (n \hbox{ times}).
$$
Its action can be extended to certain non-smooth classes of plurisubharmonic (psh) functions (cf., \cite{bedfordtaylor:dirichlet}). 
  In particular, for a psh function $u$ which is locally bounded, $(dd^cu)^n$ is well-defined as a positive measure.  
  
  If $K$ is compact and { nonpluripolar}, we define the \emph{complex equilibrium measure of $K$}  as $(dd^c V_K^*)^n$.  We also call it the \emph{(complex) Monge-Amp\`ere measure of $K$}. This is a positive measure supported on $K$.
  
For $L$-regular sets, the relationship between the higher order smoothness of $V_K$ and geometric properties of $K$ is not completely understood, except in a few special cases.  It is not known whether $V_K$ is smooth if {  $K$ is the closure of a bounded domain and} the boundary of $K$ is smooth or even real analytic.  It is known that if $K$ is the disjoint union of {  the closures of} finitely many strictly pseudoconvex domains with smooth boundary, then $V_K$ is $C^{1,1}$  \cite{guan:regularity}. 

However, the extremal function has particularly nice properties when $K$ is the closure of a bounded, smoothly bounded, \emph{strictly lineally convex} domain $D\subset\CC^n$.  Then $V_K$ is smooth on $\CC^n\setminus K$, as a consequence of  Lempert's results (\cite{lempert:metrique}, \cite{lempert:intrinsic},
 \cite{lempert:symmetries}).    He showed that there is a smooth foliation of $\CC^n\setminus K$ by holomorphic disks on which $V_K$ is harmonic (\emph{extremal disks}).

If $K$ is a convex body in $\RR^n\subset\CC^n$, it was shown in \cite{burnslevmau:exterior} that as long as $\partial K$ in $\RR^n$ does not contain parallel line segments, there is a continuous foliation of $\CC^n\setminus K$ by extremal disks.  For a \emph{symmetric} convex body, the existence of such a foliation was proved earlier in \cite{baran:plurisubharmonic} by different methods.

The existence of extremal disks through each point of $\CC^n\setminus K$ ($K$ a real convex body) was obtained in \cite{burnslevmau:pluripotential} by an approximation argument using Lempert theory, and it was shown that these disks  must be contained in complexified real ellipses (\emph{extremal ellipses}).
 An important tool used in this study was a real geometric characterization of such ellipses, which was derived from a variational description of the extremal disks.  
The goal of this paper is to establish further properties of $V_K$ by studying its foliation in more detail.   We begin in the next section by recalling basic properties of $V_K$ and its associated extremal ellipses that will be used in what follows. 

In section 3, we study the smoothness of $V_K$.  Results are proved in $\RR^2\subset\CC^2$ as the geometric arguments work only in dimension 2.  For a convex body $K\subset\RR^2$ we first {  show that at certain points of $\CC^2\setminus K$, $V_K$ is pluriharmonic} (and therefore smooth).  At other points, we use the  foliation structure of $V_K$ by extremal ellipses to study its smoothness.  We derive geometric conditions on extremal ellipses that ensure smoothness of the foliation, under the  assumption that the real boundary $\partial K$ (i.e., the boundary of $K$ as a subset of $\RR^2$) is sufficiently smooth.  We also give simple examples to illustrate what happens when these conditions fail.  Two types of ellipses are considered separately:  
\par (1) extremal ellipses intersecting $\partial K$ in exactly two points; and \par (2) extremal ellipses intersecting  $\partial K$ in exactly three points.\\  
This accounts for most ellipses; those that remain are contained in a subset of $\CC^2$ of real codimension 1. 
\vskip6pt
{  \noindent {\bf Theorem \ref{thm:smoothness}.} {\it Let $K\subset\RR^2\subset\CC^2$ be a convex body whose boundary $\partial K$ is $C^r$-smooth ($r\in\{2,3,...\}\cup\{\infty,\omega\}$).  Then $V_K$ is $C^r$ on $\CC^2\setminus K$ except for a set of real dimension at most 3.}}
\vskip6pt

Finally, in section 4, we study the complex equilibrium measure of a convex body $K\subset\RR^n\subset\CC^n$, $n\geq 2$.  If a compact set $K\subset\CC^n$ has the foliation property, we can use a ``transfer of mass'' argument to relate its complex equilibrium measure to $(dd^c\rho_K^+)^n$, where $\rho_K$ denotes the Robin function of $K$ and $\rho_K^+ = \max\{\rho_K,0\}$.  The measure $(dd^c\rho_K^+)^n$ is in fact the equilibrium measure of the \emph{Robin indicatrix of $K$}, $K_{\rho}:=\{\rho_K\leq 0\}$, and the relation is given in terms of the \emph{Robin exponential map}, first defined in \cite{burnslevmau:exterior}. 

\vskip6pt
{  \noindent {\bf Theorem \ref{thm:r2}.} {\it Let $K\subset\RR^n$ be a convex body with unique extremals.  Then for any $\phi$ continuous on $K$, 
$$
\int \phi(dd^cV_K)^n\ =\ \int(\phi\circ F)(dd^c\rho_K^+)^n.
$$} 
\vskip6pt

\noindent  Here $F$ denotes the extension of the Robin exponential map as a continuous function from $\partial K_{\rho}$ onto $K$.  A preliminary step is to prove a version of this (Theorem \ref{thm:r1}) for $K=\bar D$, the closure of a smoothly bounded, strongly lineally convex domain $D$.}

\section{Background}

In this section we recall essential properties of extremal functions and foliations associated to convex bodies.


The following properties of extremal functions are well-known.

\begin{theorem}\label{thm:2.1}
\begin{enumerate}
\item Suppose $K_1\subset\CC^n$ and $K_2\subset\CC^m$ are compact sets.  Then for $(z,w)\in\CC^{n+m}\setminus K_1\times K_2$ we have 
\begin{equation} \label{eqn:2.1a} V_{K_1\times K_2}(z,w) = \max\{V_{K_1}(z),V_{K_2}(w)\}.\end{equation}
\item Let $K\subset\CC^n$ be compact and let $P=(P_1,...,P_n):\CC^n\to\CC^n$ be a polynomial map of degree $d$, $P_j = \hat P_j + r_j$ with $\hat P_j$ homogeneous of degree $d$ and $\deg(r_j)<d$.  Suppose $\hat P^{-1}(0)=\{0\}$  where $\hat P=(\hat P_1,...,\hat P_n)$.  Then for all $z\in \CC^n$, 
\begin{equation}\label{eqn:2.1b}
V_{P^{-1}(K)}(z)=d\cdot V_K(p(z))
\end{equation}
\end{enumerate}
\end{theorem}

\begin{proof}
See e.g. Chapter 5 of \cite{klimek:pluripotential}.
\end{proof}
If $L$ is an affine change of coordinates then (\ref{eqn:2.1b}) shows that  $V_{L(K)}(L(z)) = V_K(z)$.

\bigskip

Next, let $K$ be a convex body.  We summarize the essential properties of extremal curves for $V_K$.

\begin{theorem}
 \begin{enumerate}
  \item Through every point $z\in\CC^n\setminus K$ there is either a complex ellipse $E$ with $z\in E$ such that $V_K$ restricted to $E$ is harmonic 
on $E\setminus K$, or there is a complexified real line $L$ with $z\in L$ such that $V_K$ is harmonic on $L\setminus K$.
\item For $E$ as above, $E\cap K$ as above is a real ellipse  inscribed in $K$, i.e., for its given eccentricity and orientation, it is the ellipse with largest area  
completely contained in $K$; if $L$ is as above, $L\cap K$ is the longest line segment (for its given direction) completely contained in $K$.
\item Conversely, suppose $C_T\subset K$ is a real inscribed ellipse (or line segment) with maximal area (or length) as above.  Form $E$ (or $L$) by complexification (i.e., find the unique complex algebraic curve of degree $\leq 2$ containing $C_T$.)  Then $V_K$ is harmonic on $E\setminus C_T$ (resp. $L\setminus C_T$). 
 \end{enumerate}
\end{theorem}

\begin{proof}
See Theorem 5.2 and Section 6 of \cite{burnslevmau:exterior}.
\end{proof}

The ellipses and lines discussed above have parametrizations of the form
\begin{equation}\label{eqn:2.1}
F(\zeta) = a + c\zeta + \frac{\bar c}{\zeta},
\end{equation}
$a\in\RR^n$, $c\in\CC^n$, {  $\zeta \in \CC$ with $V_K(F(\zeta))=\bigl|\log|\zeta|\bigr|$. (As usual, $\bar c$ denotes the component-wise complex conjugate of $c$.)}  These are higher dimensional analogs of the classical Joukowski function $\zeta\mapsto\frac{1}{2}(\zeta + \frac{1}{\zeta})$.

A curve parametrized as in (\ref{eqn:2.1}) is the image of the circle $\{(z_1,z_2)\in\CC^2:z_1^2+z_2^2=1\}$ under the affine map $$\CC^2\ni(z_1,z_2)\mapsto a+z_12\Re(c)-z_22\Im(c)\in\CC^n.$$
Hence $F$ parametrizes an ellipse in $\CC^n$ if $\Im(c)\neq 0$, otherwise (for $c\in\RR^2$) it gives the complex line $\{a+\lambda c: \lambda\in\CC\}$.   For convenience, we usually consider both cases together by regarding the complex lines to be degenerate ellipses with infinite eccentricity.  These algebraic curves will be referred to as  \emph{extremal curves} or \emph{extremal ellipses}, including the degenerate case.  

From part 3 of the above theorem, one can see that an extremal curve for $V_K$ may not be unique for a given eccentricity and orientation if $K$ contains parallel line segments in its boundary $\partial K$ (as a boundary in $\RR^n$), as it may be possible to translate the curve and obtain another extremal.  On the other hand, if no such line segments exist (e.g. if $K$ is strictly convex) then extremal curves are unique.

The following  was shown in \cite{burnslevmau:exterior}.
\begin{theorem} \label{thm:2.2}
If $K\subset\RR^n$ is a convex body such that all its extremal curves are unique, then these curves give a continuous foliation of $\CC\PP^n\setminus K$ by analytic disks such that the restriction of $V_K$ to any leaf of the foliation is harmonic on $\CC^n$. \qed
\end{theorem}

In the above result we {  are considering} $\CC\PP^n=\CC^n\cup H_{\infty}$ via the usual identification of homogeneous coordinates $[Z_0:Z_1:\cdots:Z_n]$ with the affine coordinates $(z_1,...,z_n)$ given by $z_i=Z_i/Z_0$ when $Z_0\neq 0$, and $H_{\infty}=\{Z_0=0\}$. An analytic disk which is a leaf of the foliation is precisely `half' of an extremal ellipse. 
Letting $\bar\Delta=\{|\zeta|\leq 1\}$ denote the closed unit disk in $\CC$ and $\hat\CC$ the Riemann sphere, a leaf of the foliation may be given by $F(\hat\CC\setminus\bar\Delta)$, with $F$ as in (\ref{eqn:2.1}) extended holomorphically to infinity via $F(\infty)=[0:c_1:\cdots:c_n]=:[0:c]$.

Analytic disks through conjugate points $[0:c]$ and $[0:\bar c]$ at $H_{\infty}$ (called \emph{conjugate leaves} in \cite{burnslevmau:pluripotential}) are the two `halves' of an extremal ellipse, and fit together along the corresponding real inscribed ellipse in $K$. 

The bulk of the proof of Theorem \ref{thm:2.2} consisted in verifying that two extremal ellipses can only intersect in the set $K$, hence they are disjoint in $\CC^n\setminus K$.  This was done on a case by case basis using the geometry of real convex bodies. 

That these ellipses foliate $\CC\PP^n\setminus K$ continuously was obtained as a by-product of the approximation techniques used to prove their existence.   This was to approximate $K$ by a decreasing sequence $\{K_i\}\searrow K$ of strongly convex, conjugation invariant bodies in $\CC^n$ with real-analytic boundary.  For such sets $K_j$, Lempert theory gives the existence of a smooth foliation of $\CC^n\setminus K_j$ by analytic disks such that the restriction of $V_{K_j}$ to each disk is harmonic.  It was also verified in \cite{burnslevmau:exterior} that these foliations extend smoothly across $H_{\infty}$ in local coordinates.  In the limit {  as $j\to\infty$, they converge to a continuous foliation parametrized by $H_{\infty}$.}

We remark that $H_{\infty}$ is a natural parameter space for leaves of the foliation by recalling its real geometric interpretation.  Two ellipses 
$$
\zeta\mapsto a+b\zeta+\bar b/\zeta,\quad \zeta\mapsto  a' +b'\zeta +\bar{b'}/\zeta 
$$
intersect the same point $c=[0:b]=[0:b']\in H_{\infty}$ if and only if $b=\lambda b'$ for some {  $\lambda\in\CC^*=\CC\setminus \{0\}$. }  Writing $\lambda=re^{i\psi}$ ($r>0$) and putting $\zeta=e^{i\theta}$ the parametrizations become
\begin{align*}
e^{i\theta}&\mapsto a+2(\Re(b)\cos\theta-\Im(b)\sin\theta)\ \hbox{ and} \\ e^{i\theta}&\mapsto a'+ 2r(\Re(b)\cos(\theta+\psi)-\Im(b)\sin(\theta+\psi)).
\end{align*}
As $\theta$ runs through $\RR$ these parametrizations trace real ellipses in $\RR^n$ related by the translation $a-a'$ and the scale factor $r$, but with  the same eccentricity and orientation.  If $a=a'$ and $|\lambda|=1$ we get a reparametrization of the same ellipse.

Given a parameter $c\in H_{\infty}$, write $a=a(c)$ and $b=b(c)$ where $\zeta\mapsto a+b\zeta+\bar b/\zeta$ ($b=(b_1,...,b_n)$) is an extremal ellipse for the eccentricity and orientation given by $c\in H_{\infty}$.  When $b_1\neq 0$, we may reparametrize the ellipse so that $b_1\in(0,\infty)$.  Put $c_j=\frac{b_j}{b_1}$ and $\rho(c)=b_1$. We then write an extremal as 
\begin{equation}\label{eqn:1.}
\zeta\mapsto a(c) +\rho(c)\Bigl((1,c_2,...,c_n)\zeta+(1,\bar c_1,...,\bar c_n)/\zeta\Bigr) \ = \ a(c) + \rho(c)\Bigl(c\zeta + \frac{\bar c}{\zeta}\Bigr)
\end{equation}
(slightly abusing notation in the last expression).  When extremals are unique, $a(c)$ and $\rho(c)$ are uniquely determined by $c$, so by Theorem \ref{thm:2.2} are continuous functions.  (Note that this is only valid locally, i.e.,  when $b_1\neq 0$.)


\section{Smoothness of $V_K$ in $\CC^2$} \label{sec:3}

We specialize now to a compact convex body $K\subset\RR^2\subset\CC^2$, and $\partial K$ will then denote the boundary in $\RR^2$. Denote coordinates in $\CC^2$ by $z=(z_1,z_2)$, and use $x=(x_1,x_2)$ when restricting to $\RR^2$.    In analyzing the extremal ellipses associated to $K$, we will employ elementary geometric arguments which do not directly generalize to higher dimensions.

\subsection{Points at which $V_K$ is pluriharmonic}

From classical potential theory in one complex variable we have the well-known formula $$V_{[-1,1]}(\zeta)=\log|h(\zeta)| \quad  (\zeta\not\in[-1,1]),$$   where $h(\zeta)=\zeta+\sqrt{\zeta^2-1}$ is the inverse of the Joukowski function (c.f., \cite{ransford:potential}).  Hence if  
$S=[-1,1]\times[-1,1]\subset\RR^2\subset\CC^2$ is the square centered at the origin, then by   (\ref{eqn:2.1a}), 
\begin{equation}\label{eqn:3.1}
V_S(z)=\max\{\log|h(z_1)|,\log|h(z_2)|\}.
\end{equation}
On $\CC^2\setminus S$ this is the maximum of two pluriharmonic functions.  A continuous foliation for $V_S$ is given by extremal ellipses for $S$ centered at the origin \cite{baran:plurisubharmonic}.


\begin{lemma} \label{lem:3.1}
Suppose $C$ is an extremal curve for the square $S$, and $z=(z_1,z_2)\in C$.  Let $i=1$ or $2$.  Then $C$ intersects both of the lines $z_i=\pm 1$ (where $i=1$ or $2$) if and only if $V_S(z)=\log|h(z_i)|$.
\end{lemma}
	
\begin{proof} 
Take $i=1$; the proof when $i=2$ is identical.  We have a parametrization $z=(a_1,a_2)+\rho((1,c_2)t+(1,\bar c_2)/t)\in C$. Note that since $C$ is extremal and intersects the lines $z_1=\pm 1$ tangentially, by symmetry the midpoint of the ellipse lies on the line $z_1=0$; so $a_1=0$. 

We verify that $\rho=\frac{1}{2}$.  
Since $C$ intersects the lines $z_1=\pm 1$ there exist $\phi_1,\phi_2\in[-\pi,\pi]$ such that
\begin{align*}
1 \ & = \  \rho(e^{i\phi_1}+e^{-i\phi_1}) = 2\rho\cos\phi_1, \\
-1 \ & = \ \rho(e^{i\phi_2}+e^{-i\phi_2}) = 2\rho\cos\phi_2
\end{align*}
 Either equation immediately implies that $\rho\geq \frac{1}{2}$, and the reverse inequality follows since the real points of $C$ lie in $S$.  Hence $z_1$ is given by the Joukowski function:   $z_1=\frac{1}{2}(t+\frac{1}{t})=h^{-1}(t)$ for $t\in\CC$ and
$$
V_S(z) = \log|t| = \log|h\circ h^{-1}(t)| = \log|h(z_1)|.
$$

Conversely, suppose $C$ does not intersect, say, $z_1=1$.  We want to show that $V_S(z)>\log|h(z_1)|$.  Now the real points of $C$, which are contained in $S$, tangentially intersect the line  $z_1=1-\epsilon$ for some $\epsilon>0$.   At a point $(z_1,z_2)\in C$, the parametrization of $C$ yields $z_1=a_1 + \rho(t+\frac{1}{t})$.  The $z_1$-coordinates of real points of $C$, given by $t=e^{i\theta}$ ($\theta\in \RR$)  must satisfy
$$
a_1 + 2\rho\cos\theta\in[-1,1-\epsilon], \quad \forall \theta.
$$
We show that $\rho\geq\frac{1}{2}$ cannot hold.  Suppose it does;  then we must have  $\cos\theta\in[-1-a_1,1-\epsilon-a_1]=:I$.  If $a_1<0$, then $\cos\pi=-1\not\in I$, a contradiction. But on the other hand, if $a_1\geq 0$ then $\cos 0=1\not\in I$, which is also a contradiction.  Hence $\rho <\frac{1}{2}$.  Finally, a calculation yields  
$$\log|h(z_1)| = \log|h(\rho(t+\frac{1}{t}))| = \log|2\rho t|<\log|t|=V_S(z_1).$$  
\end{proof}

The above lemma shows that if $z\in\CC^2\setminus S$ lies on an extremal ellipse for $S$ that does not intersect all four sides, then $V_S$ is pluriharmonic in a neighborhood of $z$.

\begin{theorem}\label{thm:3.2}
Suppose $K\subset\RR^2$ contains a pair of parallel line segments.  Suppose $C$ is an extremal curve of $K$ that intersects $\partial K$ in the interior of these two line segments and in no other points.  Then for any $z\in C\setminus K$, $V_K$ is pluriharmonic in a neighbourhood of $z$.
\end{theorem}

\begin{proof}
First, we simplify the situation using Theorem \ref{thm:2.1} and the fact that pluriharmonicity is unaffected by linear transformations.  Hence we may assume that the parallel line segments lie on the lines $x_2=1$ and $x_2=-1$ and that $C$ is centered at the origin.    By rescaling and translating the $x_1$-axis, we may further assume that $K\subset S$, $S=[-1,1]\times[-1,1]$ as above, and that $C$ intersects $\partial K$ in the two points $(\alpha,1)$ and $(-\alpha,-1)$ where $0<\alpha<1$.  

Write $C_{T}=\{F(e^{i\theta}):\theta\in\RR\}$ for the real ellipse contained in $K$, where $F(t) = \rho(bt+\bar b/t)$ is the parametrization of $C$.  By elementary topology in $\RR^2$ there exists $\epsilon>0$ such that for any $s\in[-\epsilon,\epsilon]$, the translated sets $(s,0)+C_T$ are contained in $K$.  

By construction, $C$ is extremal for $S$ as well as for $K$, so for any $z\in C\setminus\RR^2 $, we have $V_K(z)= V_S(z)$.  We show that for any sufficiently close point $z'$ we also have $V_K(z')=V_S(z')$.  

\begin{figure} 
\begin{center}
\includegraphics[scale=0.4]{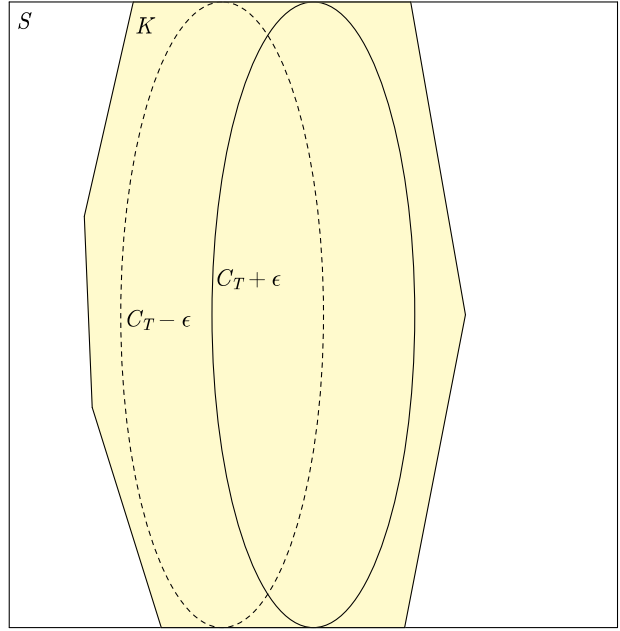}    
\end{center}
\caption{Proof of Theorem \ref{thm:3.2}}
\end{figure}

Consider an extremal ellipse $C'$ for $S$ containing a point $z'$ that is given by the parametrization $t\mapsto b't +\bar{b'}/t$.  By the continuity of the foliation for $V_S$ given by extremal ellipses centered at the origin, then 
given $\epsilon>0$ there is $\delta>0$ such that
$|z-z'|<\delta$ implies $|b-b'|<\epsilon/2$ and $|t_0-t_0'|<\epsilon/2$ where 
$$
z =ct_0+\bar c/t_0 \quad \hbox{ and } \quad z'= c't_0'+\bar{c'}/t_0'.
$$
Since $K\subset S$, to show that $C'$ is extremal for $K$ we only need to verify that $C'_T=(C'\cap S)\subset K$. Let $x=(x_1,x_2)\in C'_T$ and define $y_1$ by the condition that $y=(y_1,x_2)\in C_T$ is the closest point to $x$ with the same second coordinate.  We have 
\begin{eqnarray*}
x_1\ =\ b_1'e^{i\theta}+\bar{b_1'}e^{-i\theta}
& =&   b_1e^{i\theta}+\bar{b_1}e^{-i\theta} + (b_1'-b_1)e^{i\theta}+(\bar{b_1'}-\bar b_1)e^{-i\theta} \\ 
&=& y_1+(b_1'-b_1)e^{i\theta}+(\bar{b_1'}-\bar b_1)e^{-i\theta},
\end{eqnarray*}
where $y=(y_1,y_2)$ for some $y\in C_T$.  
  Hence $|x_1-y_1|\leq 2|b_1'-b_1|<\epsilon$.  So $C_T'$ is contained in the convex hull of $C_T-(\epsilon,0)$ and $C_T+(\epsilon,0)$, which in turn is contained in $K$. (See Figure 1.)  Therefore $C_T'$ is an extremal curve through $z'$ for both $S$ and $K$, and $V_K(z')=V_S(z')$ follows.  
  
  Applying the previous lemma, $V_K(z') = V_S(z')=\log|h(z_2')|$ for all $z'=(z_1',z_2')$ with $|z-z'|<\delta$.  So $V_K$ is pluriharmonic in a neighborhood of $z$.
\end{proof}

\begin{remark}\rm 
If $C$ is an extremal ellipse that does not intersect $\partial K$ in a pair of parallel line segments, then it is unique for its value of $c\in H_{\infty}$.   If $K$ contains parallel line segments elsewhere, we may get rid of this parallelism by modifying $K$ slightly (e.g. shaving off a thin wedge along one of the line segments).  This can be done without affecting $C$ and nearby extremals.  Hence in studying the local behaviour of extremal ellipses near a unique extremal $C$, we may assume  that uniqueness holds globally, and that extremal ellipses give a continuous foliation of $\CC^n\setminus K$ (by Theorem \ref{thm:2.2}).
\end{remark}
 
 We now turn to study the smoothness of the foliation at points on unique extremals.

\subsection{Extremal ellipses meeting $\partial K$ in two points}

As before, write $C$ to denote an extremal ellipse for $V_K$, $F$ its parametrization, and $C_T=\{F(e^{i\theta}):\theta\in\RR\}$ its trace on $K$.

\smallskip

{  \noindent NOTE: From now until the end of Section \ref{sec:3}, coordinates in $\RR^2\subset\CC^2$ will be denoted by $(x,y)$.  }

\bigskip

We will assume in what follows that $\partial K$ is at least $C^2$.  For a point $a\in\partial K$, denote by $T_a(\partial K)$ the tangent line to $\partial K$ that passes through $a$.

\begin{proposition}\label{prop:ex1}
Let $C_T\cap\partial K=\{a,b\}$.  If $\partial K$ is smooth at $a,b$, then the tangent lines $T_a(\partial K)$ and $T_b(\partial K)$ are parallel.
\end{proposition}

\begin{proof}
Let $\bv_a,\bv_b$ be unit vectors parallel to $T_a(\partial K)$, $T_b(\partial K)$ respectively.  We may assume they are oriented so 
so that $\bv_a\cdot\bv_b\geq 0$. Suppose $\bv_a\neq\bv_b$.  Then  take any unit vector $\bv$ for which 
$$\bv_a\cdot \bv_b < \bv \cdot \bv_b < 1 = \bv_b \cdot \bv_b.$$
For $t > 0$ sufficiently small, the translated ellipse $C_{T,t} := C_T + t\bv$ is then contained in the interior of $K$ so that $C_{T,t}$ can be expanded to an ellipse with the same orientation and eccentricity as $C_T$.  This contradicts the fact that $C_T$ is extremal. 
\end{proof}

\bigskip

Let us start now by fixing an extremal curve $C$, with $C_T\subset K$, corresponding to 
a fixed value $c = c_0\in H_{\infty}$ and $\rho(c_0) =\rho_0$.  The parametrization of $C$ may be written as 
\begin{equation}\label{eqn:ref}
\zeta \mapsto (x_0,y_0) + \rho_0\Bigl( (1,c_0)\zeta + (1,\bar c_0)/\zeta\Bigr).
\end{equation}
(In the above equation, we identify $c_0$ with its representation in local coordinates, i.e., as a complex number $c_0\in\CC$).
\begin{figure}
\begin{center}
\includegraphics[scale=1.2]{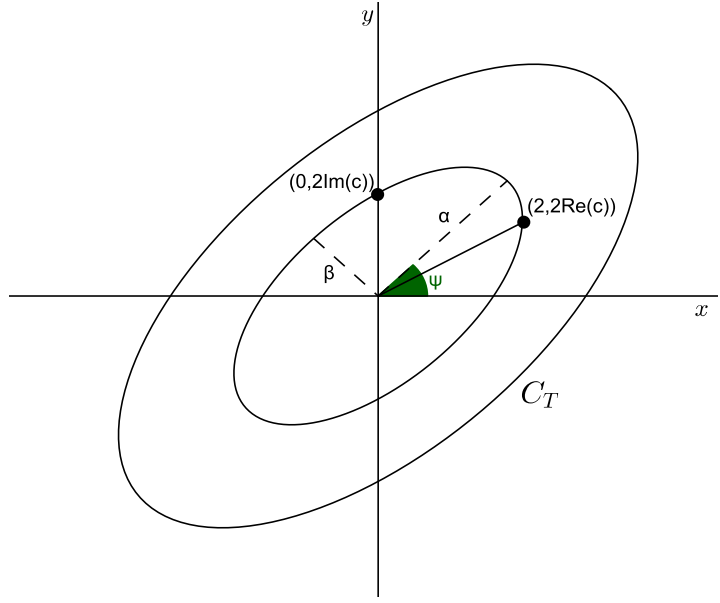}
\end{center}
\caption{The parameters $\alpha,\beta,\psi$ and their relation to $c$ in local coordinates. The smaller ellipse is the reference ellipse given by setting {  $(x_0,y_0)=(0,0)$ and $\rho_0=1$ in (\ref{eqn:ref}).}  Then $\alpha,\beta$ are the lengths of its axes and $\psi$ is the angle between the axis of length $\alpha$ and the horizontal axis.  Scaling by $\rho_0$ (and translating to $(x_0,y_0)$) yields the extremal ellipse $C_T$.   Note that in local coordinates, a reference ellipse is the unique ellipse, for a given eccentricity and orientation, that is centered at the origin and tangent to the vertical line $x=2$.}
\label{fig:2a}
\end{figure}

For convenience, we will use a more natural parametrization of $C_T$ from the point of view of real geometry, i.e., $C_T = F(\theta) =(F_1(\theta), F_2(\theta))$ where 
 \begin{eqnarray}
F_1(\theta) &=& \rho_0[ \alpha\cos\theta \cos \psi - \beta \sin \theta \sin \psi] + x_0, \label{eqn:3.2} \\
F_2(\theta) &=& \rho_0[\alpha \cos \theta \sin \psi + \beta \sin \theta \cos \psi] + y_0, \label{eqn:3.3}
\end{eqnarray}
Here $\alpha$, $\beta$ and $\psi$ incorporate the parameter $c\in H_{\infty}$; in local coordinates, they are real-analytic functions of $c$.  {  Precisely, $c$ determines $\psi$ and $\gamma:=\frac{\beta}{\alpha}$, which are scale-invariant parameters.}  (See Figure \ref{fig:2a} for the explicit geometry.)   By rotating coordinates, it is no loss of generality to assume that $\alpha\sin\psi\neq 0$, which we will assume in what follows. 


Differentiating (\ref{eqn:3.2}) and (\ref{eqn:3.3}), we have 
\begin{eqnarray}
F_1'(\theta) &=& \rho_0[-\alpha\sin\theta \cos\psi - \beta \cos\theta \sin\psi] + x_0, \label{eqn:3.4}\\
F_2'(\theta) &=& \rho_0[-\alpha\sin\theta \sin\psi + \beta \cos\theta \cos\psi] + y_0. \label{eqn:3.5}
\end{eqnarray}

Let $a = F(\theta_0)$ and $b = F(\theta_1)$ be the points of intersection with $\partial K$.  Write  
$K = \{(x, y)\in\RR^2 : r(x, y) = 0\}$ with $\nabla r(a), \nabla r(b) \neq (0, 0)$.  By Proposition \ref{prop:ex1}, $\theta_1 = \theta_0 + \pi$. We rotate
coordinates and normalize $r$ so that $\nabla r(a) = (0, 1)$ and $\nabla r(b) = (0,-\lambda)$ 
with $\lambda> 0$. Since
\begin{equation*}
\nabla r(a)\cdot F'(\theta_0)= \nabla r(b)\cdot F'(\theta_1)=0,
\end{equation*}
we have
$$
F_2'(\theta_0) = F_2'(\theta_1) =0.$$

Take $a = (0, 0)$ and write $\partial K$  
near $a$ as $\partial K = \{(s, \eta(s)) : |s| <\epsilon \}$ with $\eta(0) = \eta'(0) = 0$. Now we
consider variations in $s$, so that we consider the point 
$a(s) := (s, \eta(s))$
on $\partial K$. This determines the normal $\nabla r(a(s))$ and an ``antipodal''  point
$b(s) = (x(s), y(s))\in\partial K$ such that 
{  $\nabla r(b(s)) = \lambda\nabla r(a(s))$ for some
$\lambda=\lambda(s) < 0$.} We define the parameter $\theta_0(s)$ via the defining relation
\begin{equation*}
r_x(s,\eta(s))\frac{\partial F_1}{\partial\theta}(\theta_0(s)) + r_y(s,\eta(s))\frac{\partial F_2}{\partial\theta}(\theta_0(s)) = 0
\end{equation*}
and this defines $\theta_1(s) := \theta_0(s) +\pi$. We also write the center as 
$$(x_0(s), y_0(s)) = [a(s) + b(s)]/2 = (s/2 + x(s)/2, \eta(s)/2 + y(s)/2).$$
 Allowing $\rho$ and $c$ 
(i.e., $\rho$, $\gamma$ and $\psi$) to vary, we now consider $F$ as a function
$$F(\theta) = F(\theta_0(s), x_0(s), y_0(s),\rho,\alpha,\beta,\psi).$$

Consider the equations
\begin{eqnarray*}
A(s, \rho,\gamma,\psi) &:=& s - F_1(\theta_0(s), x_0(s), y_0(s), \rho,\gamma,\psi), \\
B(s, \rho,\gamma,\psi) &:=& \eta(s)- F_2(\theta_0(s), x_0(s), y_0(s),\rho,\gamma,\psi).
\end{eqnarray*}
We get a mapping $(s, \rho,\gamma,\psi)\mapsto(A,B)$ near $(0,\rho_0,\gamma_0,\psi_0)$ where
$\rho_0=\rho(c_0)$ for the parameter $c = c_0$ corresponding to $\gamma_0$ and  $\psi_0$; i.e., at $s = 0$.
Thus we have $A(0, \rho_0,\gamma_0,\psi_0) = B(0, \rho_0,\gamma_0,\psi_0) = 0$,
\begin{equation*} 
r_x(s,\eta(s))\frac{\partial F_1}{\partial\theta}(\theta_0(s)) + 
r_y(s,\eta(s))\frac{\partial F_2}{\partial\theta}(\theta_0(s)) = 0, \end{equation*}
  and
  \begin{equation*}
r_x(x(s),y(s))\frac{\partial F_1}{\partial\theta}(\theta_1(s)) + 
r_y(x(s),y(s))\frac{\partial F_2}{\partial\theta}(\theta_1(s)) = 0.
\end{equation*}

We want to find conditions for which
\begin{equation} \label{eqn:3.5a}
\det \left.\begin{pmatrix}
\frac{\partial A}{\partial s} & \frac{\partial A}{\partial\rho}   \\
\frac{\partial B}{\partial s} & \frac{\partial B}{\partial\rho}
\end{pmatrix} \right|_{s=0,\rho=\rho_0} \neq 0.
\end{equation}
Then by the implicit function theorem we can solve for $s$, $\rho$ near $0$, $\rho_0$ in terms of $\gamma$, $\psi$ (i.e., $c$) near $\gamma_0$, $\psi_0$.  

We write, for simplicity, $\theta_0=\theta_0(0)$ and $\theta_1=\theta_1(0)$ so that
$$
\frac{\partial F_2}{\partial \theta}(\theta_0) = \frac{\partial F_2}{\partial \theta}(\theta_1)=0.
$$
From (\ref{eqn:3.5}), $\frac{\partial F_2}{\partial\theta}(\theta_0)=0$ says that 
\begin{equation}\label{eqn:3.6}
\alpha\sin\theta_0\sin\psi = \beta\cos\theta_0\cos\psi.
\end{equation}

We compute the entries of the matrix in (\ref{eqn:3.5a}).  Below, prime $'$ denotes differentiation with respect to $s$.
$$
\frac{\partial A}{\partial s} = 1-\frac{\partial F_1}{\partial\theta}\theta' - \frac{\partial F_1}{\partial x_0}x_0'  - \frac{\partial F_1}{\partial y_0} y_0' . 
$$
From (\ref{eqn:3.2}), $\frac{\partial F_1}{\partial x_0} = 1$ and $\frac{\partial F_1}{\partial y_0}= 0$.  Thus
\begin{equation} \label{eqn:3.7}
\frac{\partial A}{\partial s} = 1 - \frac{\partial F_1}{\partial\theta}\theta' - x_0'.
\end{equation}

Next,
$$
\frac{\partial A}{\partial\rho} = -\frac{\partial F_1}{\partial\rho} = -(\alpha\cos\theta\cos\psi - \beta\sin\theta\sin\psi).
$$
Then
$$
\frac{\partial B}{\partial s} = \eta' -\frac{\partial F_2}{\partial\theta}\theta' - \frac{\partial F_2}{\partial x_0}x_0' - \frac{\partial F_2}{\partial y_0}y_0'
= \eta' - \frac{\partial F_2}{\partial \theta}\theta' - \frac{\partial F_2}{\partial y_0}
$$
since (\ref{eqn:3.3}) implies $\frac{\partial F_2}{\partial x_0} = 0$ and $\frac{\partial F_2}{\partial y_0} = 1$.  Moreover, at $s=0$, we have $\eta'(0)=0$ and $\frac{\partial F_2}{\partial\theta}(\theta_0)=0$ so that 
$$
\left.\frac{\partial B}{\partial s}\right|_{s=0}  \ = \   
\left.\frac{\partial F_2}{\partial y_0}\right|_{s=0} y_0'(0).
$$
But $y_0(s) = \frac{1}{2}(\eta(s) + y(s))$ so that
$$
y_0'(0) = \frac{1}{2}\bigl( \eta'(0) + y'(0) \bigr) = 0.
$$
Thus $\left.\frac{\partial B}{\partial s}\right|_{s=0} = 0$.

On the other hand, we claim that if $\alpha\sin\psi\neq 0$, then  $\left.\frac{\partial B}{\partial\rho}\right|_{s=0}\neq 0$.  By (\ref{eqn:3.3}),
$$
\frac{\partial B}{\partial\rho} = -\frac{\partial F_2}{\partial\rho} = -(\alpha\cos\theta\sin\psi + \beta\sin\theta\cos\psi).
$$
If on the contrary, $\left.\frac{\partial B}{\partial\rho}\right|_{s=0} =  0$, then using (\ref{eqn:3.6}),
\begin{eqnarray*}
\alpha\sin\theta_0\sin\psi &=& \beta\cos\theta_0\cos\psi \hbox{ and} \\
\alpha\cos\theta_0\sin\psi &=& -\beta\sin\theta_0\cos\psi.
\end{eqnarray*}
Multiplying the top equation by $\sin \theta_0$ and the bottom one by $\cos\theta_0$ and adding, we obtain $\alpha\sin\psi=0$, a contradiction.
Thus (\ref{eqn:3.5a}) holds precisely when $\left.\frac{\partial A}{\partial s}\right|_{s=0}\neq 0$. 

\bigskip

 We now show that
\begin{equation} \label{eqn:3.9a}
\left.\frac{\partial A}{\partial s}\right|_{s=0} \ =\  \frac{1}{2} + \Bigl(\frac{\partial F_1}{\partial\theta}(\theta_0)\Bigr)^2\frac{r_{xx}(0,0)}{y_0} - \frac{1}{2}\frac{r_{xx}(0,0)r_y(x(0),y(0))}{r_{xx}(x(0),y(0))r_y(0,0)}.
\end{equation}
To see this, recall first that $r_x(0,0)=0$ and $r_y(0,0)=1$.  Moreover,
$$
r_x(s,\eta(s))\frac{\partial F_1}{\partial\theta}(\theta(s)) + r_y(s,\eta(s))\frac{\partial F_2}{\partial\theta}(\theta(s))=0;
$$
differentiating this equation with respect to $s$ we get
$$
(r_{xx}+r_{yx}\eta')\frac{\partial F_1}{\partial\theta} + r_x\frac{\partial^2 F_1}{\partial\theta^2}\theta' + (r_{xy}+r_{yy}\eta')\frac{\partial F_2}{\partial\theta} + r_y\frac{\partial^2F_2}{\partial\theta^2}\theta'=0.
$$
Now at $s=0$, $\eta'(0)=\frac{\partial F_2}{\partial\theta}(\theta_0) =r_{x}(0,0)=0$; moreover, writing $F_2:= J+y_0$ we see that $\frac{\partial^2F_2}{\partial\theta^2}=-J$ so that
$$
\frac{\partial^2F_2}{\partial\theta^2}(\theta_0) = -J(\theta_0) = y_0 - F_2(\theta_0) = y_0.
$$
Hence
$$
r_{xx}(0,0)\frac{\partial F_1}{\partial\theta}(\theta_0)+ y_0\theta'(0) = 0
$$
and so
\begin{equation} \label{eqn:3.10a}
\theta'(0) = \frac{-r_{xx}(0,0)\frac{\partial F_1}{\partial\theta}(\theta_0)}{y_0}.
\end{equation}

To compute/rewrite $x_0'(0)$, we use the fact that $\nabla r(s,\eta(s)) = \lambda\nabla r(x(s),y(s))$.  This implies the relation
$$
r_x(x(s),y(s))\cdot r_y(s,\eta(s)) - r_y(x(s),y(s))\cdot r_x(s,\eta(s))=0.
$$
Differentiate this with respect to $s$, and set $s=0$:
\begin{eqnarray*}
&& \Bigl(r_{xx}(x(0),y(0))\cdot x'(0)\ + \  r_{xy}(x(0),y(0))\cdot y'(0)\Bigr)r_y(0,0) \\
&& \hskip3.5cm - \  \Bigl(r_{xx}(0,0)+r_{xy}(0,0)\cdot\eta'(0)\Bigr)r_y(x(0),y(0)) \ =\  0.
\end{eqnarray*}
Here we have used the fact(s) that $r_x(0,0) = r_x(x(0),y(0))=0$.  But we also have $\eta'(0)=y'(0)=0$,  $r_y(0,0)=1$, and $r_y(x(0),y(0))=-\lambda$; and so 
$$
r_{xx}(x(0),y(0))\cdot x'(0) + \lambda r_{xx}(0,0) = 0;
$$
i.e.,
$$
x'(0) = \frac{-\lambda r_{xx}(0,0)}{r_{xx}(x(0),y(0))}.
$$
Now $x_0(s)=\frac{1}{2}(s+x(s))$, so $x_0'(0)=\frac{1}{2}(1+x'(0))$, and so
\begin{equation} \label{eqn:3.10}
x_0'(0) = \frac{1}{2} - \frac{1}{2}\frac{\lambda r_{xx}(0,0)}{r_{xx}(x(0),y(0))}.
\end{equation}
Plugging (\ref{eqn:3.10a}) and (\ref{eqn:3.10}) into (\ref{eqn:3.7}) yields (\ref{eqn:3.9a}).

\bigskip

We now analyze the situation when $\left.\frac{\partial A}{\partial s}\right|_{s=0}=0$.  From (\ref{eqn:3.9a}), this occurs precisely when
\begin{eqnarray} \label{eqn:3.12}
\frac{\left(\frac{\partial F_1}{\partial\theta}(\theta_0)\right)^2}{y_0} &=& \frac{1}{2}\Bigl( \frac{1}{r_{xx}(0,0)} \ + \ \frac{r_y(x(0),y(0))}{r_y(0,0)r_{xx}(x(0),y(0))}\Bigr) \\
&=& \frac{1}{2}\Bigl( \frac{1}{r_{xx}(0,0)} \ + \ \frac{r_y(x(0),y(0)}{r_{xx}(x(0),y(0))}    \Bigr). \nonumber
\end{eqnarray}
Since $r_x(0,0)=r_x(x(0),y(0))=0$ and $r_y(0,0)=1$, the right-hand side of (\ref{eqn:3.12}) is the average of the radii of the osculating circles of $\partial K$ at $a=(0,0)$ and $b=(x(0),y(0))$.  

\smallskip

We claim that the left-hand side of (\ref{eqn:3.12}) is the average of the radii of the osculating circles of $C_T$ at $a$ and $b$.  To see this, note from (\ref{eqn:3.3}) and (\ref{eqn:3.4}) (and $\frac{\partial^2F_2}{\partial\theta^2}(\theta_0)=y_0$),  that 
$$
\frac{\left(\frac{\partial F_1}{\partial \theta}(\theta_0)\right)}{y_0} = \frac{\rho(-\alpha\sin\theta_0\cos\psi - \beta\cos\theta_0\sin\psi)^2}{-(\alpha\cos\theta_0\sin\psi + \beta\sin\theta_0\cos\psi)}.
$$
Let us rotate coordinates so that $\sin\psi=1$; i.e., $\psi=\frac{\pi}{2}$, and assume $\theta_0=0$.  Then
$$
\frac{\left(\frac{\partial F_1}{\partial\theta}(\theta_0)\right)^2}{y_0} = \frac{\rho\beta^2}{\alpha}.
$$
On the other hand, $C_T$ now has the parametrization 
$$
 F_1(\theta) = -\rho\beta\sin\theta + x_0, \quad 
 F_2(\theta) = \rho\alpha\cos\theta + y_0,  
$$ 
and the curvature of $C_T$ as a function of $\theta$ is 
$$
\kappa(\theta) = \frac{1}{\rho}\cdot
\frac{\alpha\beta}{(\beta^2\cos^2(\theta)+\alpha^2\sin^2\theta)^{3/2}}.
$$
At $\theta=0,\pi$ we get 
$$
\kappa(0) = \kappa(\pi) = \frac{\alpha}{\beta^2\rho}
$$
as claimed (precisely, 
$\frac{1}{2}\Bigl(\frac{1}{\kappa(0)}+\frac{1}{\kappa(\pi)}\Bigr) = \frac{1}{2}\cdot\frac{2\beta^2\rho}{\alpha} = \frac{\rho\beta^2}{\alpha}$.)

\begin{remark}\rm
\begin{itemize}
\item[(i)] Note that the radii of the osculating circles for $\partial K$ at the points $a,b$ are at least as large as those for $C_T$ since $C_T$ is inscribed in $\partial K$.  Thus the condition $\left.\frac{\partial A}{\partial s}\right|_{s=0}=0$ fails if the curvature of $C_T$ is strictly less than that of $\partial K$ at either $a$ or $b$.
\item[(ii)] A degenerate ellipse (i.e. line segment) occurs when $\beta=0$. A careful examination of the preceding calculations shows that (\ref{eqn:3.5a}) always holds in this case.
\end{itemize}
\end{remark}

{  When $\partial K$ is $C^r$, the implicit function theorem shows that $s$ and $\rho$ can be solved in terms of $\gamma,\psi$ (equivalently, $c\in H_{\infty}$) as $C^r$ functions.  This implies that locally, the center $(x_0(s),y_0(s))$ is a $C^r$ function of $c$, and must therefore coincide with $a(c)$ given in equation (\ref{eqn:1.}).  Similarly, the scale factor $\rho(c)$ is also a  $C^r$ function of $c$.   Altogether, this shows that the foliation of extremal ellipses near $C_T$ is $C^r$.}

The smoothness of the foliation at a point $z\in\CC^n\setminus K$ in turn implies the smoothness of $V_K$ at $z$, as the partial derivatives $\partial/\partial z_j$ may be computed explicitly in terms of foliation parameters using the chain rule.   We summarize this in the following theorem.

\begin{theorem} \label{thm:3a}
Suppose $z\in\CC^2\setminus K$ lies on an extremal ellipse $C$ for $K$ with the following properties: 
\begin{enumerate}
\item The intersection $\partial K\cap C$ is exactly two points.
\item $\partial K$ is $C^r$ ($r\geq 2$) in a neighborhood of $\partial K\cap C$.
\item For at least one of these intersection points, the curvature of $C_T$ is strictly greater than the curvature of $\partial K$ at this point.
\end{enumerate}
   Then $V_K$ is $C^r$ in a neighborhood of $z$. \qed
\end{theorem} 

\def\calC{\mathcal{C}}

\begin{remark} \label{rmk:3a}
\rm
Parameters $c\in H_{\infty}$ corresponding to extremal curves for $V_K$ that intersect $\partial K$ in two points but do not satisfy the curvature condition (3) form a set of real dimension at most $1$.  One way to see this is to consider the collection of all extremal ellipses $\calC_a$ that intersect $\partial K$ at a point $a$.  Then $\calC_a$ is parametrized by a subset of (real) dimension 1 in $H_{\infty}$, and $\bigcup_{a\in\partial K}\calC_a = \CC^2$.  For each $a$, however, there is at most one ellipse parameter $c_a$ for which the curvature of the extremal ellipse coincides with the curvature of $\partial K$.  Now as $a$ varies smoothly over the curve $\partial K$, $c_a$ varies smoothly over a one-dimensional subset of  $H_{\infty}$. Hence $\bigcup_{a\in\partial K}\{c_a\}$ is at most $1$-dimensional. \end{remark}

\begin{example}\label{ex:realdisk} \rm
Consider $K$ to be the real unit disk,  $K=\{(x,y)\in\RR^2\subset\CC^2: x^2+y^2\leq 1\}$.  Then by Lundin's formula \cite{lundin:extremal},
$$
V_K(z)=\frac{1}{2}\log^+(|z_1|^2+|z_2|^2+|z_1^2+z_2^2-1|).
$$
On $\CC^2\setminus K$ this function is nonsmooth precisely on the complex ellipse $C$ given by  $z_1^2+z_2^2=1$.  In this case $C_T=\partial K$, so trivially the curvatures are equal.
\end{example}

\subsection{Extremal ellipses meeting $\partial K$ in three points}

Suppose an extremal ellipse $C$ meets $\partial K$ in three points: $$C\cap\partial K = \{(x_1,y_1),(x_2,y_2),(x_3,y_3)\}.$$
In this case, the ellipse is necessarily non-degenerate.

As before, we parametrize $C_T$ via $F(\theta)=(F_1(\theta),F_2(\theta))$ as in equations (\ref{eqn:3.2}) and (\ref{eqn:3.3}).  For $i=1,2,3$ define $\theta_i\in[0,2\pi)$ via $F(\theta_i)=(x_i,y_i)$.

We want to analyze this set-up under variations of $\beta$ and $\psi$ (i.e., $c$).  
We will assume that coordinates have been chosen so that $\partial K$ can be represented as a graph over either $x$ or $y$ near each point $(x_i,y_i)$; in particular, we will assume that for each $i=1,2,3$ there exist smooth functions $\eta_i$ with
\begin{itemize}
\item $\partial K = \{(x,y): y=\eta_1(x)\}$ near $(x_1,y_1)$;
\item $\partial K = \{(x,y): x=\eta_2(y)\}$ near $(x_2,y_2)$; and 
\item $\partial K = \{(x,y): y=\eta_3(x)\}$ near $(x_3,y_3)$.
\end{itemize}

Now consider variations in $\beta$ and $\psi$, and consider the variables $\rho,x_0,y_0$ and $x_i,y_i,\theta_i$, $i=1,2,3$ to be dependent on these variations.  In total, there are $12$ dependent variables.

We will eliminate eight of these variables.  Using the functions $\eta_i$, we immediately eliminate $y_1$, $x_2$ and $y_3$.  We now proceed to eliminate $\theta_i$, $i=1,2,3$.  We use the fact that the intersection $C_T\cap\partial K$ is tangential at each point $(x_i,y_i)$.  For $i=1$, this says that
$$
(F_1'(\theta_1),F_2'(\theta_1)) = \lambda(1,\eta_1'(x_1));
$$
i.e., $F_1'(\theta_1)\eta_1'(x_1) = F_2'(\theta)$.  Explicitly, using equations (\ref{eqn:3.2}) and (\ref{eqn:3.3}) we obtain
\begin{equation}\label{eqn:3.12a}
\tan(\theta_1) = \frac{\beta(\cos\psi + \eta_1'(x_1)\sin\psi)}{\sin\psi - \eta_1'(x_1)\cos\psi}.
\end{equation}
Locally, we may take the principal branch of arctan (that gives angles in $[0,2\pi)$) to obtain the function $\theta_1(x_1,\beta,\psi)$, and hence eliminate $\theta_1$ as a dependent variable.  Similarly we can do the same for $\theta_2$ and $\theta_3$.

The last two variables we will eliminate are $x_0$ and $y_0$.  First, define the following functions (for notational convenience we suppress their dependence on the variables $\beta,\psi,\rho,x_0,y_0$):
\begin{eqnarray*}
A_1(x_1) &=& F_1(\theta_1) - x_1, \\
B_1(x_1) &=& F_2(\theta_1) - \eta_1(x_1), \\
A_2(y_2) &=& F_1(\theta_2) - \eta_2(y_2), \\
B_2(y_2) &=& F_2(\theta_2) - y_2, \\
A_3(x_3) &=& F_1(\theta_3) - x_3, \hbox{ and} \\
B_3(x_3) &=& F_2(\theta_3)-\eta_3(x_3).
\end{eqnarray*}
The geometric condition that the points $(x_i,y_i)$ are intersections of $C_T$ with $\partial K$ says that $A_i=B_i=0$, $i=1,2,3$. 

Define $S_1(\theta):=F_1(\theta)-x_0$ and $S_2(\theta):=F_2(\theta)-y_0$; then $A_3=B_3=0$ says that 
$$
x_0 = x_3 - S_1(\theta_3) \ \hbox{ and } \  y_0 = \eta_3(x_3) - S_2(\theta_3).
$$
Using this to eliminate $x_0,y_0$,  the system of equations reduces to
\begin{eqnarray*}
A_1 &=& S_1(\theta_1) - S_1(\theta_3) + x_3-x_1, \\
B_1 &=& S_2(\theta_1) - S_2(\theta_3) + \eta_3(x_3) - \eta_1(x_1), \\
A_2 &=& S_1(\theta_2) - S_1(\theta_3) +x_3 -\eta_2(y_2), \hbox{ and} \\
B_2 &=& S_2(\theta_2) -S_2(\theta_3) + \eta_3(x_3) - y_2.
\end{eqnarray*}
In summary, we have a map $M:(x_1,y_2,x_3,\rho,\beta,\psi)\mapsto(A_1,A_2,B_1,B_2)$ where the geometric condition that $C_T$ is inscribed in $\partial K$ implies that $M=0$.

We can solve for $x_1,y_2,x_2,\rho$ in terms of $\beta$ and $\psi$ provided the Jacobian matrix
$$
JM=\begin{pmatrix}
\frac{\partial A_1}{\partial x_1} & \frac{\partial A_1}{\partial y_2} & \frac{\partial A_1}{\partial x_3} & \frac{\partial A_1}{\partial\rho} \\
\frac{\partial B_1}{\partial x_1} & \frac{\partial B_1}{\partial y_2} & \frac{\partial B_1}{\partial x_3} & \frac{\partial B_1}{\partial\rho} \\
\frac{\partial A_2}{\partial x_1} & \frac{\partial A_2}{\partial y_2} & \frac{\partial A_2}{\partial x_3} & \frac{\partial A_2}{\partial\rho} \\
\frac{\partial B_2}{\partial x_1} & \frac{\partial B_2}{\partial y_2} & \frac{\partial B_2}{\partial x_3} & \frac{\partial B_2}{\partial\rho} 
\end{pmatrix}
$$
has nonzero determinant.  Note that $\frac{\partial A_1}{\partial y_2}=\frac{\partial B_1}{\partial y_2}= \frac{\partial A_2}{\partial x_1}=\frac{\partial B_2}{\partial x_1}=0$.


\begin{figure}
\begin{center}
\includegraphics[scale=0.15]{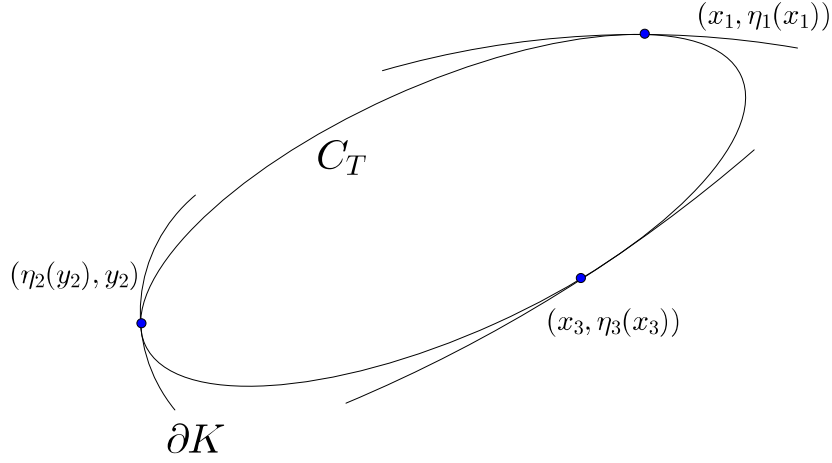}
\end{center}
\caption{Local parametrizations of $\partial K$}
\label{fig:2} 
\end{figure}

Fix an initial inscribed ellipse, and denote its parameters by $x_{i0},y_{i0}$, ($i=1,2,3$) and $\rho_0$.  To simplify the computations, without loss of generality we may assume that 
$$
\eta_1'(x_1)\bigr|_{x_1=x_{10}}=0, \ \  \eta_2'(y_2)\bigr|_{y_2=y_{20}}=0, \ \hbox{ and } \  
\eta_3'(x_3)\bigr|_{x_3=x_{30}}>0
$$
by applying a linear change of coordinates (see Figure \ref{fig:2}).  In these coordinates, the tangency of $C_T$ at its intersections with $\partial K$ says that $S_2'(\theta_1)\bigr|_{\theta_1=\theta_1(x_{10})} = 0$, and so
\begin{equation}\label{eqn:3.13a}
\left.\frac{\partial B_1}{\partial x_1}\right|_{x_1=x_{10}} = 
\left[S_2'(\theta_1)\frac{\partial\theta_1}{\partial x_1} - \eta_1'(x_1)\right]_{x_1=x_{10}} = 0;
\end{equation}
a similar argument also gives $\left.\frac{\partial A_2}{\partial y_2}\right|_{y_2=y_{20}} = 0$.  Hence
\begin{equation}\label{eqn:3.13}
\det(JM) = 
\det\begin{pmatrix}
\frac{\partial A_1}{\partial x_1} & 0 & \frac{\partial A_1}{\partial x_3} & \frac{\partial A_1}{\partial\rho} \\
0&0 & \frac{\partial B_1}{\partial x_3} & \frac{\partial B_1}{\partial\rho} \\
0&0 & \frac{\partial A_2}{\partial x_3} & \frac{\partial A_2}{\partial\rho} \\
0 & \frac{\partial B_2}{\partial y_2} & \frac{\partial B_2}{\partial x_3} & \frac{\partial B_2}{\partial\rho} \\
\end{pmatrix} = \frac{\partial A_1}{\partial x_1} \frac{\partial B_2}{\partial y_2} \det(M_1),
\end{equation}
where $M_1 = \begin{pmatrix} 
\frac{\partial B_1}{\partial x_3} & \frac{\partial B_1}{\partial\rho}\\
\frac{\partial A_2}{\partial x_3} & \frac{\partial A_2}{\partial\rho}  \end{pmatrix}$.

We derive conditions under which each factor on the right-hand side of (\ref{eqn:3.13}) is nonzero.  First,
\begin{eqnarray} \label{eqn:3.14}
\frac{\partial A_1}{\partial x_1} &=&  S_1'(\theta_1)\frac{\partial \theta_1}{\partial x_1} - 1, \hbox{ and}\\  \label{eqn:3.15}
\frac{\partial B_2}{\partial y_2} &=& S_2'(\theta_2)\frac{\partial\theta_2}{\partial y_2} - 1.
\end{eqnarray}

\def\be{\mathbf{e}}

We analyze $\det(M_1)$.  For convenience, translate coordinates to the origin, i.e., put $(x_{30},y_{30})=(0,0)$.  We then rotate coordinates as follows.
  Let $\be_1=(1,0)$ and $\be_2=(0,1)$ denote the standard basis in $\RR^2$; let $R_{\alpha} = \begin{pmatrix} \cos\alpha & \sin\alpha \\
  -\sin\alpha & \cos\alpha \end{pmatrix}$, where $\alpha\in(0,\frac{\pi}{2})$ is given by $\tan\alpha = \eta_3'(0)$; now let $\be_1' = R_{\alpha}^{-1}\be_1$ and  $\be_2' = R_{\alpha}^{-1}\be_2$.  Let $(\cdot)_{R_{\alpha}}$ denote coordinates and matrices written with respect to this new basis, e.g.,
  $$
  (\tilde x,\tilde y)_{R_{\alpha}} = \tilde x\be_1' + \tilde y\be_2'.
  $$
  In these coordinates, the common tangent to $\partial K$ and $C_T$ at $0=(0,0)_{R_{\alpha}}$ has no  second (i.e., $\be_2'$) component.  This says that $\tilde\eta_3'(0)=0$ where $(\tilde x,\tilde\eta_3(\tilde x))_{R_{\alpha}} = (x,\eta_3(x))$.
  
  In what follows, tilded quantities (e.g., $\tilde S_1$, $\tilde S_2$) denote quantities expressed with respect to rotated coordinates. We calculate that
  \begin{eqnarray*}
  R_{\alpha}\begin{pmatrix} \frac{\partial A_2}{\partial x_3} \\ \frac{\partial B_1}{\partial x_3}
  \end{pmatrix}
   &=& R_{\alpha}\begin{pmatrix}
   -1 + S_1'(\theta_3)\frac{\partial\theta_3}{\partial x_3} \\ -\eta_3'(x_3) + S_2'(\theta_3)\frac{\partial\theta_3}{\partial x_3}
   \end{pmatrix} \\
   &=& R_{\alpha}\begin{pmatrix}
   -1 + \tilde S_1'(\theta_3)\frac{\partial\theta_3}{\partial\tilde x_3} \\
   -\tilde\eta_3'(\tilde x_3) + \tilde S_2'(\theta_3)\frac{\partial\theta_3}{\partial\tilde x_3}
   \end{pmatrix}_{R_{\alpha}} \frac{\partial\tilde x_3}{\partial x_3} \\
   &=& R_{\alpha}R_{\alpha}^{-1}\begin{pmatrix}
   -1 + \tilde S_1'(\theta_3)\frac{\partial\theta_3}{\partial \tilde x_3} \\ 0
   \end{pmatrix}\cos\alpha,
  \end{eqnarray*}
  where zero in the bottom component above follows from $\tilde\eta_3'(0) = 0$, by the same argument that gave (\ref{eqn:3.13a}) earlier.
  
  Define $B_{R_{\alpha}}$ and $A_{R_{\alpha}}$ by 
  $\begin{pmatrix}A_{R_{\alpha}} \\ B_{R_{\alpha}}\end{pmatrix} := R_{\alpha}\begin{pmatrix}{\frac{\partial A_2}{\partial\rho}}\\{\frac{\partial B_1}{\partial\rho}}\end{pmatrix}$.  With $E = \begin{pmatrix} 0&1\\1&0
  \end{pmatrix}$, we have
\begin{eqnarray*}
\det(M_1) =-\det(ER_{\alpha}M_1) &=& \det\begin{pmatrix}
0 & B_{R_{\alpha}} \\
\bigl(1 - \tilde S_1'(\theta_3)\frac{\partial\theta_3}{\partial\tilde x_3}\bigr)\cos\alpha & A_{R_{\alpha}}
\end{pmatrix} \\
&=& \bigl(\tilde S_1'(\theta_3)\frac{\partial\theta_3}{\partial \tilde x_3} -1\bigr)(\cos\alpha)B_{R_{\alpha}}.
\end{eqnarray*}
Therefore, $\det(JM)\neq 0$ holds if and only if none of (\ref{eqn:3.14}), (\ref{eqn:3.15}),
\begin{eqnarray} \label{eqn:3.16} 
\tilde S_1'(\theta_3)\frac{\partial\theta_3}{\partial\tilde x_3} &=& 1, \ \hbox{ or} \\
\label{eqn:3.17} B_{R_{\alpha}} &=& 0
\end{eqnarray}
hold.

We analyze (\ref{eqn:3.14}).  Differentiating (\ref{eqn:3.2}) and (\ref{eqn:3.12a}) (the latter implicitly), we obtain
$$
1 \ = \ S_1'(\theta_1)\frac{\partial\theta_1}{\partial x_1} \ = \  
-\frac{\rho\eta_1''(x_1)}{\beta}\cdot\frac{\cos^2\theta_1(\sin^2\theta_1\cos\psi + \beta\cos\theta_1\sin\psi)}{(\sin\psi - \eta_1'(x_1)\cos\psi)^2}\, .
$$
Using the fact that $\eta_1'(x_1)=0$, we can simplify this to 
$$
\frac{1}{\eta_1''(x_1)} \ = \ -\rho\beta\left(
\frac{\sin\theta_1\cos^2\theta_1}{\tan\psi\sin\psi} + 
\frac{\beta\cos^3\theta_1}{\sin\psi}\right),
$$
and also simplify (\ref{eqn:3.12a}) to $\tan\psi = \frac{\beta}{\tan\theta_1}$; hence $\sin\psi = \frac{\beta\cos\theta_1}{\sqrt{\beta^2\cos^2\theta_1+\sin^2\theta_1}}$.  Eliminating $\psi$, we obtain
\begin{eqnarray*}
\frac{1}{\eta_1''(x_1)} &=& -\frac{\rho}{\beta}\sqrt{\beta^2\cos^2\theta_1 + \beta^2\cos^2\theta_1(\sin^2\theta_1+ \beta^2\cos^2\theta_1)} \\
&=& -\frac{\rho}{\beta}(\sin^2\theta_1+\beta^2\cos^2\theta_1)^{\frac{3}{2}}.
\end{eqnarray*}
Now, note that $\kappa_{\partial K}(x_1) = \eta_1''(x_1)$ and $\kappa_{C_T}(x_1) = \frac{\beta}{\rho}(\sin^2\theta_1 + \beta^2\cos^2\theta_1)^{-\frac{3}{2}}$, where $\kappa_{\partial K}(x_1)$ (resp., $\kappa_{C_T}(x_1)$) denotes the curvature of $\partial K$ (resp., $C_T$) at the point $x_1$.  Hence
$$
S_1'(\theta_1)\frac{\partial\theta_1}{\partial x_1}=1  \iff 
\kappa_{C_T}(x_1) = \kappa_{\partial K}(x_1).
$$
By (\ref{eqn:3.14}), the above condition in turn is equivalent to  $\frac{\partial A_1}{\partial x_1}=0$.

Similar calculations as above show that 
\begin{eqnarray*}
S_2'(\theta_2)\frac{\partial\theta_2}{\partial y_2} = 1 &\iff& \kappa_{C_T}(x_2) = \kappa_{\partial K}(x_2), \hbox{  and} \\
\tilde S_1'(\theta_3)\frac{\partial\theta_3}{\partial \tilde x_3} = 1 &\iff&  \kappa_{C_T}(x_3) = \kappa_{\partial K}(x_3).
\end{eqnarray*} 
Using (\ref{eqn:3.14}), (\ref{eqn:3.15}) and (\ref{eqn:3.16}), we obtain the following geometric criterion:
\begin{itemize}
\item[($\star$)] {\it If $\det JM\neq 0$ then the curvature of $C_T$ is strictly greater than the curvature of $\partial K$ at each of the three intersection points $(x_i,y_i)$, $i=1,2,3$.}
\end{itemize}

It remains to show that (\ref{eqn:3.17}) always fails.  For if not, then
$$
0=B_{R_{\alpha}} = (S_2(\theta_1)-S_2(\theta_3))\cos\alpha - (S_1(\theta_2) - S_1(\theta_3))\sin\alpha,
$$
i.e., $$\tan\alpha = \frac{S_2(\theta_1)-S_2(\theta_3)}{S_1(\theta_2)-S_1(\theta_3)}.$$  However, interpreting each side of the above equation geometrically, $\tan\alpha=\eta_3'(x_3)$ is the slope of $\partial K$ at $(x_3,y_3)$, while on the other hand, $\frac{S_2(\theta_1)-S_2(\theta_3)}{S_1(\theta_2)-S_1(\theta_3)}$ is minus the slope of the line connecting $\bv_1$ and $\bv_2$, where for $i=1,2$, $\bv_i$ denotes the closest point to $(x_3,y_3)$ that lies on the tangent line to $\partial K$ through $(x_i,y_i)$.  Hence $\tan\alpha>0>\frac{S_2(\theta_1)-S_2(\theta_3)}{S_1(\theta_2)-S_1(\theta_3)}$, a contradiction.  So (\ref{eqn:3.17}) fails.

This shows that if condition ($\star$) holds, then $\det(JM)\neq 0$ and locally we may solve for $(x_1,y_2,x_3,\rho)$ as functions of $(\beta,\psi)$, and hence for $\rho$ and $a_0=(x_0,y_0)$ as functions of $c$.  If $\partial K$ is $C^r$ ($r\geq 2$), then the foliation for $V_K$ is locally $C^r$ at points on $C$, and hence so is $V_K$.  

We have proved the following result.

\begin{theorem}\label{thm:3b}
Suppose $z\in\CC^2\setminus K$ lies on an extremal ellipse $C$ for $K$ with the following properties. 
\begin{enumerate}
\item The intersection $\partial K\cap C$ is exactly three points.
\item $\partial K$ is $C^r$ ($r\geq 2$) in a neighborhood of $\partial K\cap C$.
\item The curvature of $C_T$ is strictly greater than the curvature of $\partial K$ at each of the three intersection points.
\end{enumerate}
   Then $V_K$ is $C^r$ in a neighborhood of $z$. \qed
\end{theorem}

Similar reasoning as in Remark \ref{rmk:3a} shows that the parameters for which condition (1) holds but condition (3) fails form a subset of $H_{\infty}$ of at most one real dimension.

\subsection*{Extremal ellipses meeting $\partial K$ in more than three points}
The same reasoning also shows that the parameters for extremal ellipses meeting $\partial K$ in at least four points form a lower dimensional subset.  Note that $V_K$ may not be smooth across such ellipses.
\begin{example} \label{ex:4pts} \rm 
For the unit square $S=[-1,1]\times[-1,1]$, we have by Lemma \ref{lem:3.1} that $V_S(z)=\log|h(z_1)|=\log|h(z_2)|$ if $z=(z_1,z_2)$ lies on an extremal ellipse that intersects all four sides.  In $\CC^2\setminus S$, $V_S$ is not smooth precisely on the set where $|h(z_1)|=|h(z_2)|$, which is a submanifold of real dimension 3. 
\end{example}

The results of this section may be summarized in the following theorem.

\begin{theorem}\label{thm:smoothness}
Let $K\subset\RR^2\subset\CC^2$ be a convex body whose boundary $\partial K$ is $C^r$-smooth ($r\in\{2,3,...\}\cup\{\infty,\omega\}$).  Then $V_K$ is $C^r$ on $\CC^2\setminus K$ except for a set of real dimension at most 3.
\end{theorem}

\begin{proof}
A point at which $V_K$ is not smooth must lie on an extremal ellipse $C$ that satisfies one of the following conditions:
\begin{itemize}
\item $C$ satisfies Properties (1)  but not (3) in Theorem \ref{thm:3a};
\item $C$ satisfies Properties (1) but not (3) in Theorem \ref{thm:3b}; or
\item $C$ meets $\partial K$ in at least 4 points.
\end{itemize}
A collection of ellipses that satisfies one of the above conditions forms {  at most} a (real) one-parameter family; so the union of these ellipses is {  at most} a real 3-dimensional set.
\end{proof}

That $V_S$ in Example \ref{ex:4pts} is a maximum of smooth functions is an instance of a more general phenomenon.  Given a symmetric convex set $K$, suppose $z\in C$ where $C$ is an extremal ellipse that intersects $K$ in four points, and suppose that the curvature of $C_T$ is strictly greater than that of $\partial K$ at these points.  We enlarge $K$ in  two different ways to obtain convex sets $K_1, K_2$ with the following properties:
\begin{itemize}
\item $K=K_1\cap K_2$.
\item $C$ is an extremal ellipse for each of $K_1,K_2$.
\item On some neighborhood of $C\setminus K$, $V_{K_j}$ is smooth for each $j=1,2$ and \begin{equation}\label{eqn:max1} V_K=\max\{V_{K_1},V_{K_2}\}.\end{equation}
\end{itemize}
Theorem \ref{thm:smoothness} gives the smoothness of the $V_{K_j}$'s in a neighborhood of $C\setminus\RR^2$.  Figure \ref{fig:4a} illustrates this method when $\partial K$ is given by $x^4+y^4=1$.  Since $V_K$ is the maximum of two functions, it is not a priori smooth across $C$ (and we expect non-smoothness in general).  For $\partial K$ given by $x^{2n}+y^{2n}=1$ (where $n>1$), the ellipses that intersect $\partial K$ in four points form a real 3-dimensional set in $\CC^2$ in which we expect $V_K$ to be non-smooth.

 Note that the same sort of argument works in other cases, e.g. $C_T$ intersects $\partial K$ in more than four points and/or $K$ is not symmetric.  One can show that $V_K$ is locally a maximum of smooth functions by considering enlargements of $K$ to convex sets whose boundaries each meet $C_T$ in three points, and then taking their extremal functions.
 
 The same argument can also be used to get local smooth approximations: given $\epsilon>0$ and $z_0\in\CC^2\setminus K$, one can construct a convex set $K_{\epsilon}$ with the property that for all $z$ in some neighborhood of $z_0$, $V_{K_{\epsilon}}$ is smooth and $V_{K_{\epsilon}}(z)\leq V_K(z)\leq V_{K_{\epsilon}}(z)+\epsilon$.

\begin{figure}
\begin{center}
\includegraphics[scale=0.3]{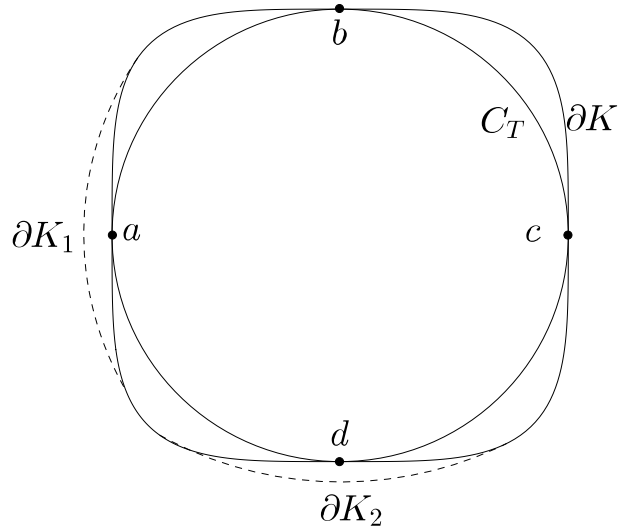}
\end{center}
\caption{The construction of $K_1$ and $K_2$ by locally modifying $\partial K$.    The real ellipse $C_T=C\cap K$ intersects $\partial K_1$ in $\{b,c,d\}$ and $\partial K_2$ in $\{a,b,c\}$.     
Equation (\ref{eqn:max1}) holds near $C$ because an extremal ellipse for $K$ whose parameters are sufficiently close to those of $C$ is also an extremal ellipse for one of the $K_j$'s.   }
\label{fig:4a}
\end{figure}

For any compact convex set $K$, it seems plausible that one can make a finite number of local boundary modifications as in Figure \ref{fig:4a} to remove the `bad' conditions listed in the proof of Theorem \ref{thm:smoothness}, at least q.e.\footnote{Recall that a property holds q.e.\;=\;\emph{quasi-everywhere} if it holds everywhere outside a (possibly nonempty) pluripolar set.}  This motivates the following conjecture.
\begin{conjecture*}
Let $K\subset\RR^2$ be a convex body with smooth boundary.
\begin{enumerate}
\item There is a finite collection $\{K_j\}$ of convex bodies with the property that $V_{K_j}$ is smooth q.e. on $\CC^2\setminus K_j$ for each $j$, and $V_K=\max_j V_{K_j}$. 
\item Given $\epsilon>0$ there is a convex body $K_{\epsilon}$ such that $V_{K_{\epsilon}}$ is smooth q.e. on $\CC^2\setminus K_{\epsilon}$ and $|V_K(z)-V_{K_{\epsilon}}(z)|<\epsilon$ for all $z\in\CC^2$.
\end{enumerate} 
\end{conjecture*}

\begin{remark}\rm 
The above conjecture is (trivially) true for the real disk, whose extremal function is smooth away from $z_1^2+z_2^2=1$.  It is not known if there is a real convex set whose extremal function is smooth everywhere on its complement.
\end{remark}


\section{The complex equilibrium measure and the Robin exponential map}

In this section, we relate the complex equilibrium (or Monge-Amp\`ere) measure of a convex set $K$ to that of its Robin indicatrix, defined below.

Given a compact set $K\subset\CC^n$, the Robin function $\rho_K$ of $K$ is the logarithmically homogeneous, psh function given by $\limsup_{|\lambda|\to\infty}[V_K^*(\lambda z)-\log|\lambda|]$.
The \emph{Robin indicatrix of $K$} is the set given by
\begin{equation}\label{eqn:4.1}
K_{\rho} = \{z\in\CC^n: \rho_K(z)\leq 0\}.
\end{equation}

Let $D$ be a bounded, strictly lineally convex domain with smooth boundary.  We recall some basic facts concerning  Lempert extremal curves for $V_{\overline D}$; {  i.e., holomorphic curves which foliate $\CC^n\setminus \overline D$ on which $V_{\overline D}$ is harmonic,} and the Robin indicatrix of $\overline D$  (cf.,  \cite{burnslevmau:exterior}, \cite{lempert:holomorphic}, \cite{momm:extremal}). Recall that $\Delta=\{\zeta\in\CC: |\zeta|<1\}$ is the open unit disk.

\begin{proposition}\label{prop:4.1}
 Let $K=\overline D$ where $D$ is a bounded, strongly lineally convex body in $\CC^n$ with smooth boundary.  Then
\begin{enumerate}
\item A Lempert extremal curve may be represented as $f:\CC\setminus\Delta\to\CC^n\setminus K$, with Laurent expansion
\begin{equation}\label{eqn:4.2a}
f(\zeta) = a_1\zeta + \sum_{j\leq 0}a_j\zeta^j, \quad \ a_j\in\CC^n,\ a_1\neq 0
\end{equation} 
and $V_K(f(\zeta))=\log|\zeta|$.
\item A Lempert extremal curve may be extended continuously to a map on $\partial\Delta$ with $f(\partial\Delta)\subset\partial D$.
\item A Lempert extremal curve is orthogonal to the level sets of $V_K$. Precisely, if  $z=f(\zeta)\in\CC^n\setminus K$, then the complex hyperplane $H_z$ given by
$$H_z=\{z+w:\ w\cdot tf'(\zeta)=0 \ \forall t\in\CC\} $$ is tangent to the level set of $V_K$ at $z$.
\item
If $v\in\partial K_{\rho}$ then $v=\lim_{|\zeta|\to\infty} f(\zeta)/\zeta$ for some $f$ that parametrizes an extremal curve for $V_K$.  We obtain the same extremal curve for $w\in\partial K_{\rho}$ if and only if $w=ve^{i\theta}$: in this case, $w=\lim_{|\zeta|\to\infty} g(\zeta)/\zeta$ where $g(\zeta)=f(\zeta e^{-i\theta})$. 

\item
There is a smooth diffeomorphism $F:\CC^n\setminus K_{\rho}\to\CC^n\setminus K$ such 
$\rho_K(z) = V_K(F(z))$, and for any Lempert extremal disk parametrized as in  (\ref{eqn:4.2a}), we have
$F(a_1\zeta)  = f(\zeta)$.
\end{enumerate}
\end{proposition}
The smooth diffeomorphism $F$ in (5) is called the \emph{Robin exponential map}.  The set $K_{\rho}$ together with the   complex lines through the origin may be regarded as a linearized model of $K$ and the associated foliation for $V_K$. 

By part (2), we can extend the Robin exponential map to $\partial K_{\rho}$ via $F(a_1e^{i\theta})=f(e^{i\theta})$.  Part (4) ensures that the map is well-defined.

\medskip

  Given a bounded, strictly lineally convex domain $D$ with smooth boundary, let $V=V_K$ be the extremal function of its closure $K=\overline D$.  For $\lambda\in(1,\infty)$, write 
$$
D_{\lambda} = \{z\in\CC^n: V(z)<\log\lambda\}
$$
for the sublevel sets of $V$.

\medskip

The main ingredient to relate the equilibrium measure of $K$ with that of its Robin indicatrix $K_{\rho}$ is the following ``transfer of mass'' formula. 

\begin{lemma}[Transfer to a level set] \label{lem:r1}
Let $D$ be as above, and suppose $\psi$ is a continuous function on $\CC^n \setminus D$ with the property that $\psi(f (\zeta)) = \psi(f(\zeta/|\zeta|))$ for all
Lempert extremal disks $f:\Delta\to\CC^n \setminus D$. Then for all $1<\lambda_1<\lambda_2<\infty$,
\begin{equation}\label{eqn:L4.2}
\int_{\partial D_{\lambda_1}} \psi d^cV\wedge(dd^cV)^{n-1} = \int_{\partial D_{\lambda_2}} \psi d^cV\wedge(dd^cV)^{n-1}.
\end{equation}
\end{lemma}                                          

\begin{proof}
Let us first carry out the proof under the assumption that $\psi$ is smooth. By Stokes' theorem, 
\begin{align*}
\int_{-\partial D_{\lambda_1}\cup\partial D_{\lambda_2}} \psi d^cV\wedge(dd^cV)^{n-1}\  &= \ 
 \int_{D_{\lambda_2}\setminus\bar D_{\lambda_1}} \psi(dd^cV)^n  \\
 &\  \hskip2cm + \  \int_{D_{\lambda_2}\setminus\bar D_{\lambda_1}} d\psi\wedge d^cV\wedge (dd^cV)^{n-1} \\
& := \ I + II. 
\end{align*}
Here $-\partial D_{\lambda_1}$ means that we use the opposite orientation on $\partial D_{\lambda_1}$ (i.e., the boundary orientation induced by the complement of $D$). 

 Proving the lemma is equivalent to showing that $I+II=0$.                           
      Now $I = 0$  since $(dd^cV)^n=0$ on $\CC^n \setminus D$. 
      
      Therefore, we must show that $II = 0$. 
First note that with respect to polar coordinates $z=re^{i\theta}$ in one variable, we have $d(\log|z|)=dr/r$ and $d^c\log|z|=d\theta$.
Let $f : \CC \setminus \Delta \to \CC^n \setminus D$ parametrize a Lempert disk, with $V (f (t)) =
\log |t|$, $t = re^{i\theta}$. Then $f^*d^c V = d\theta$. Also, since $\psi\circ f(re^{i\theta}) = \psi\circ f(e^{i\theta} )$, we have $f^*d\psi =
\gamma d\theta$ for some function $\gamma$; thus $f^*(d\psi\wedge d^c V ) = 0$. This says that $d\psi\wedge d^c V$ annihilates any pair of vectors 
tangent to the curve parametrized by $f$, so it  can act nontrivially only on the components that are normal to the curve.  Hence by Proposition \ref{prop:4.1}(2), $d\psi\wedge d^c V$  can act nontrivially only on pairs of vectors with components in the complex tangent space of the level sets of $V$, which is of complex dimension $n-1$.

On the other hand, $f^*dd^c V = 0$ since $V$ is harmonic along the extremal curve; so $dd^cV$ is a $(1,1)$-form that also  can act nontrivially only on pairs of vectors with components in the complex tangent space of the level sets of $V$.  Since $d\psi\wedge d^c V\wedge (dd^c V )^{n-1}$ is a smooth (alternating) $(n,n)$-form that acts only on vectors spanning a space of complex dimension $n-1$, it must be identically zero. So $II = 0$, which proves the lemma when $\psi$ is smooth.

If $\psi$ is only continuous, we first restrict it to $\partial D$ and approximate it by a sequence of smooth functions $\psi_n$, with $\psi_n\to\psi$ uniformly.  This can be done as follows.  Since $\partial D$ is smooth, locally we have a smooth diffeomorphism $\chi:U\subset\RR^{2n-1}\to\partial D$, and if $\psi$ is supported in $\chi(U)$, take $\psi_n = (\psi\circ\chi)_n\circ\chi^{-1}$ where $(\psi\circ\chi)_n$ are standard mollifications of $\psi\circ\chi$ in $\RR^{2n-1}$ with $(\psi\circ\chi)_n\to\psi\circ\chi$ uniformly on $U$.  Since $\partial D$ is compact, a general $\psi$ continuous on $\partial D$ can be mollified as above using a partition of unity, with  $\psi_n\to\psi$ uniformly.

Next, extend the functions $\psi_{n}$ from $\partial D$ to $\CC^n\setminus D$ via $\psi_n(f(\zeta))=\psi_n(f(\frac{\zeta}{|\zeta|}))$ where $f$ parametrizes a foliation disk.  The extended functions $\psi_n$  converge uniformly to $\psi$ on $\CC^n\setminus D$, and moreover are smooth on $\CC^n\setminus\partial D$ since the Lempert foliation is smooth.  Therefore (\ref{eqn:L4.2}) holds with $\psi$ replaced by $\psi_n$.   Taking the limit as $n\to\infty$, and using the uniform convergence $\psi_n\to\psi$,  yields the result for $\psi$.
\end{proof}

We now use the above lemma and a limiting procedure to transfer the Monge-Amp\`ere
measure of $K$ to the boundary of  $K_{\rho}$, the Robin indicatrix.  In the calculations that follow we will use a standard Monge-Amp\`ere formula (see e.g. 
\cite{bedfordmau:complex}): given a smooth psh function $u$ such that the boundary of the set $\{u>0\}$ is a (real) smooth hypersurface $S$, and $u^+=\max\{u,0\}$, then for any continuous function $\varphi$, 
\begin{equation}\label{eqn:4.2b}
\int\varphi (dd^cu^+)^n \ = \  \int_{\{u>0\}} \varphi(dd^cu)^{n} \ + \     \int_{S} \varphi d^cu\wedge (dd^cu)^{n-1}.
\end{equation}

 Note that by (\ref{eqn:4.2b}), the Monge-Amp\`ere measures of ${\overline D_{\lambda}}$ ($\lambda\in(1,\infty)$)  and $K_{\rho}$ are given (with $\varphi$ an arbitrary continuous function) by the formulas 
\begin{equation}\label{eqn:4.4a} \int\varphi (dd^cV_{{\overline D_{\lambda}}})^n = \int_{\partial D_{\lambda}}\varphi d^cV_K\wedge(dd^cV_K)^{n-1}\end{equation}   and
\begin{equation}\label{eqn:4.5a} \int\varphi (dd^c\rho_K^+)^n=\int_{\partial K_{\rho}}\varphi d^c\rho_K\wedge(dd^c\rho_K)^{n-1}.\end{equation}

We will also use standard convergence properties of Monge-Amp\`ere measures (see e.g. \cite{bedfordtaylor:new}).  Recall that for a sequence $\{u_j\}$ of locally bounded psh functions on a domain $D$ we have the weak-$*$ convergence of measures $(dd^cu_j)^n\to(dd^cu)^n$ for some locally bounded psh function $u$ on $D$ whenever:
\begin{enumerate}
\item $u_j\to u$ monotonically as $j\to\infty$ (i.e., $u_j\nearrow u$  a.e. or $u_j\searrow u$); or
\item $u_j\to u$ locally uniformly as $j\to\infty$.
\end{enumerate}

 \begin{theorem}   \label{thm:r1}     
Let $K=\bar D \subset\CC^n$ be the closure of a {  smoothly bounded}, strongly lineally convex domain $D$;
let $V=V_{K}$ be its Siciak-Zaharjuta extremal function, and let $\rho$ be its Robin function. Then for
any continuous $\phi$ on $\partial K$,
\begin{equation} \label{eqn:r1}
\int \phi (dd^cV)^{n} = \int_{\partial K_{\rho}} (\phi\circ F)d^c\rho\wedge(dd^c\rho)^{n-1},
\end{equation}
where $K_{\rho} = \{\rho \leq 0\}$ is the Robin indicatrix, and $F:\CC^n\setminus K_{\rho}\to\CC^n\setminus K$ is the Robin exponential map.
\end{theorem}

\noindent\emph{Proof.}\ 
Using the Lempert foliation, we extend $\phi$ continuously to $\CC^n \setminus D$ by the formula 
$\phi(f (\zeta)) :=    \phi(f (\zeta/|\zeta|))$, for all Lempert disks $f : \Delta\to\CC^n\setminus D$.  Next, we have $V_{\overline{D}_{\lambda}}\nearrow V$, so $(dd^cV_{\overline D_{\lambda}})^n\to(dd^cV)^n$ in the weak-$*$ convergence
of measures.  But for this particular $\phi$, 
$
\int \phi (dd^cV_{\overline D_{\lambda}})^n  \ = \  \int_{\partial D_{\lambda}} \phi d^cV\wedge(dd^cV)^{n-1}
$
is constant in $\lambda$ by Lemma \ref{lem:r1}, and hence the equality 
$$\int\phi(dd^cV)^n = \int_{\partial D_{\lambda}}\phi d^cV\wedge(dd^cV)^{n-1}$$ holds for all  $\lambda>1$.

To prove the theorem, it suffices to show that the right-hand side of the above equation
converges to the right-hand side of (\ref{eqn:r1}) as $\lambda\to \infty$. To see this, let $\lambda t = z$; then
\begin{align*}
& \int_{z\in\partial D_{\lambda}}\phi(z)d^c_zV\wedge(dd^c_zV)^{n-1} \\
& \qquad = \ \int_{t\in\lambda^{-1}(\partial D_{\lambda})}\!\! \phi(t\lambda)d^c_t(V(t\lambda)-\log|\lambda|)\wedge(dd^c_t(V(t\lambda)-\log|\lambda|))^{n-1}.
\end{align*}
For clarity, in the above lines the dependence (of derivatives and integrals) with respect to the variable 
$z$ or $t$ has been made explicit.  (Note that $\lambda d_z= d_t$, $\lambda d_z^c= d_t^c$.) 

Away from the origin, the convergence $V_{\lambda} (t) := V (t{\lambda}) - \log |\lambda|\longrightarrow \rho(t)$
is uniform as $\lambda\to \infty$. Hence we get the weak-* convergence $(dd^c (\max{V_{\lambda} , 0}))^n \to (dd^c (\max{\rho, 0}))^n$,  i.e.,
$$
\int_{\lambda^{-1}(\partial D_{\lambda})}\varphi d^cV_{\lambda}\wedge(dd^cV_{\lambda})^{n-1}\to \int_{\partial K_{\rho}}\varphi d^c\rho\wedge(dd^c\rho)^{n-1}
$$
for any continuous function $\varphi$.
                                                                              
     In particular, let $\psi$ be the continuous function given by 
     \begin{equation}\label{psidef} \psi( \zeta a ) = (\phi\circ F )(a) \ \hbox{for all} \ 
a\in\partial K_{\rho} , \ \zeta\in (1, \infty). \end{equation} Then
\begin{equation}\label{eqn:r2}
\int_{\lambda^{-1}\partial D_{\lambda}}\psi d^cV_{\lambda}\wedge(dd^cV_{\lambda})^{n-1}\longrightarrow
\int_{\partial K_{\rho}}\psi d^c\rho\wedge(dd^c\rho)^{n-1} \hbox{ as } \lambda\to \infty.
\end{equation}

We want to replace $\psi(t)$ on the left-hand side by $\phi(t\lambda)$ for $\lambda$ sufficiently large.  To do this, we will prove a couple of lemmas before returning to the main proof.

\begin{lemma}\label{lem:4.4}
There is a constant $C>0$ such that  $|a\xi - f_a(\xi)|<C$ for all $a\in\partial K_{\rho}$, $|\xi|>1$. 
\end{lemma}

\begin{proof}
Let $f_a$ denote the parametrization of the Lempert extremal curve given by $\xi\mapsto a\xi + \sum_{j\leq 0}a_j\xi^j$.  The family of maps $g_a(\zeta)=a/\zeta - f_a(1/\zeta)$ ($a\in\partial K_{\rho}$) are holomorphic on the unit disk $\Delta$ and uniformly bounded there by the maximum principle, since $f_a(\partial\Delta)\subset K$.  In other words, there is a uniform bound $|a/\zeta - f_a(1/\zeta)|<C$ independent of $a$ and $\zeta\in\Delta$; now identify $\xi=1/\zeta$.
\end{proof}


\begin{lemma} \label{lem:r2}
Let $V=V_K$ be the Siciak-Zaharjuta extremal function for a compact set $K$. Suppose
$V$ is continuous, and that $\CC^n \setminus K$ is foliated continuously by Lempert extremal disks.
Suppose $\varphi$ is a continuous function on $\CC^n \setminus K$ such that for any leaf $f : \CC\setminus\Delta\to\CC^n \setminus K$ of the foliation,
$\varphi(f (\lambda)) = \varphi(f (r\lambda))$ for all $r \in (\frac{1}{|\lambda|},\infty )$. 

Then given $C > 0$, $\epsilon > 0$, there exists $R = R(C) > 0$
such that $|\varphi(z) - \varphi(z')| < \epsilon$ whenever $|z|, |z'| > R$ and $|z - z' | < C$.
\end{lemma}

\begin{proof}

Without loss of generality, $K \subset B(0, 1)$, the unit ball. Then $|V (z) - V (z' )| \leq \omega( |z-z'|/|z| )$ where $\omega$ 
is the modulus of continuity of $V$ on $B(0, 2)$ (see e.g., Lemma 4.5 \cite{bloomlevmau:robin}).


 Hence given $\eta$
such that $\omega(\eta)<\delta$ for some prescibed $\delta>0$, then by choosing $R_0 > C/\eta$ , we have that if 
$|z|, |z'| > R_0$ and $|z- z'| < C$, then
\begin{equation} \label{eqn:r4}
|V(z)-V(z')|<\delta. 
\end{equation}
We also have
\begin{equation}\label{eqn:4.7}
z=f_a(\xi)= a\xi+\sum_0^{\infty}b_k\xi^{-k}, \ \ z'=f_{a'}(\xi') = a'\xi'+\sum_0^{\infty}b_k'\xi'^{-k},
\end{equation}
with $V (z) = \log|\xi|$, $V (z') = \log|\xi'|$ , and $a, a' \in D_{\rho}$. 

Without loss of generality, we  
assume $\xi = |\xi|$ (by reparametrization). If $z, z'$ satisfy (\ref{eqn:r4}), then $e^{-\delta} < \frac{|\xi'|}{|\xi|} < e^{\delta}$. 
Hence $\frac{\xi'}{\xi}=e^{\alpha+i\theta}$ for some $|\alpha| < \delta$, 
$\theta\in [0, 2\pi)$.

 Using the series expansions in (\ref{eqn:4.7}) to estimate $z-z'$, 
we have
$$
C\geq |z-z'| \geq \left|a\xi-a'\xi'\right| - M = |\xi||a-a'e^{\alpha+i\theta}| - M
$$
for some constant $M$ (which may be obtained using the uniform bound in the previous lemma).
Hence $|a-a'e^{\alpha+i\theta}|\leq \frac{1}{|\xi|}(C+M)$.  The right-hand side can be made smaller than $\delta>0$ by taking $|\xi|>\frac{C+M}{\delta}$.

Now take $\epsilon>0$ as given by hypothesis, and choose $\tilde\delta>0$ such that 
\begin{equation}\label{eqn:4.10d}
 \hbox{if $t,t'\in \partial K$ and $|t-t'|<\tilde\delta$, then $|\varphi(t)-\varphi(t')|<\epsilon$.}
\end{equation}  
By continuity of the foliation for $V_K$, we can choose $\delta>0$ such that  
\begin{equation}\label{eqn:4.11d}  
\hbox{$|f_c(1)-f_{c'}(1)|<\tilde\delta$ whenever $|c-c'|<\delta$ and $c,c'\in\partial K_{\rho}$.}  
\end{equation} 
Choose $R>R_0$ sufficiently large that if $|z|>R$, then $V_K(z)>\frac{C+M}{\delta}+1$.  Now given $|z|,|z'|>R$, we have $z=f_a(\xi)$ and $z'=f_{a'}(\xi')$ with $V_K(z)=\log|\xi|$ and $V_K(z')=\log|\xi'|$; and $|z-z'|<C$ implies $|a-a'e^{\alpha+i\theta}|<\delta$.  Finally, 
\begin{align*}
 |\varphi(z) - \varphi(z')| &= |\varphi(f_a(\xi)) - \varphi(f_{a'}(\xi'))| \\
 &= |\varphi(f_a(|\xi|)) - \varphi(f_{a}(e^{-i\theta}|\xi'|))| \\
 &= |\varphi(f_a(1)) - \varphi(f_{a}(e^{-i\theta})| \\
 &=  |\varphi(f_a(1)) - \varphi(f_{e^{i\theta}a}(1)| < \epsilon,
\end{align*}   
where we apply (\ref{eqn:4.10d}) and (\ref{eqn:4.11d}) in the last line.  This concludes the proof.
\end{proof}

\begin{proof}[End of the proof of Theorem \ref{thm:r1}]  By Lemma \ref{lem:4.4} , there is $C>0$ such that  $$|a\xi-f_a(\xi)|<C  \ \hbox{ for all } |\xi|>1,\ a\in\partial K_{\rho}.$$  Given $t$, choose $s\in\CC$ such that $t=sa$ for some $a\in\partial K_{\rho}$.  Then for $\lambda>1$ sufficiently large (chosen so that $|s\lambda a|,|f_a(s\lambda)|>R$ for all $a\in\partial K_{\rho}$, where $R=R(C)$ is chosen as in Lemma \ref{lem:r2}), we have from (\ref{psidef}) 
 $$
|\psi(t)-\phi(t\lambda)|= |\psi(t\lambda)-\phi(t\lambda)| = |\psi(s\lambda a)-\phi(s\lambda a)|
 = |\phi(f_a(s\lambda)) - \phi(s\lambda a)|<\epsilon.
 $$ 
Then $$\left|\int_{\lambda^{-1}\partial D_{\lambda}}\psi(t)d^cV_{\lambda}\wedge(dd^cV_{\lambda})^{n-1}\ - \    
 \int_{\lambda^{-1}\partial D_{\lambda}}\phi(t\lambda)d^cV_{\lambda}\wedge(dd^cV_{\lambda})^{n-1}\right|<(2\pi)^n\epsilon.$$ (The factor of $(2\pi)^n$ is due to the fact that $V_{\lambda}\in L^+(\CC^n)$ implies $\int (dd^cV_{\lambda})^n=(2\pi)^n$; see e.g. \cite{bedfordtaylor:uniqueness}).  
 Since $\epsilon$ is arbitary, we obtain
\begin{equation}\label{eqn:4.7a}
\lim_{\lambda\to \infty}\int_{\lambda^{-1}\partial D_{\lambda}}\psi(t)d^cV_{\lambda}\wedge(dd^cV_{\lambda})^{n-1}\ = \  
\lim_{\lambda\to \infty}\int_{\lambda^{-1}\partial D_{\lambda}}\phi(t\lambda)d^cV_{\lambda}\wedge(dd^cV_{\lambda})^{n-1}.
\end{equation}
The expression inside the limit on the right-hand side is actually constant in $\lambda$: changing back to the variable $z$ where $\lambda t=z$, we have
\begin{eqnarray*}
\int_{\lambda^{-1}\partial D_{\lambda}}\phi(t\lambda)d^cV_{\lambda}\wedge(dd^cV_{\lambda})^{n-1} &=& \int_{\partial D_{\lambda}}\phi(z)d^cV\wedge(dd^cV)^{n-1}\\ 
&=& \int\phi(dd^cV_{\lambda})^{n}  
\end{eqnarray*}
which is independent of $\lambda>1$ by Lemma \ref{lem:r1}.  Then 
$$
\lim_{\lambda\to\infty} \int_{\partial D_{\lambda}} \phi d^cV\wedge (dd^cV)^{n-1} \ = \ 
\lim_{\lambda\to 1^+}\int\phi (dd^cV_{\lambda})^n \ = \ 
\int\phi (dd^cV_K)^n
$$
where the last equality follows by Monge-Amp\`ere convergence.  
    Finally, putting the above together with (\ref{eqn:r2}) and (\ref{eqn:4.7a}) finishes the proof. 
\end{proof}

We use Theorem \ref{thm:r1} to get a similar result for convex bodies $K\subset\RR^n$ with unique extremals.  The Robin exponential map in this  case is of the form $$F(b\zeta)=a+b\zeta+\frac{\bar b}{\zeta}, \quad b\in\partial K_{\rho},\   \zeta\in\CC\setminus\Delta.$$  (Recall that $a$ depends on $b$; more precisely, on $c=[0:b_1:\cdots:b_n]\in H_{\infty}$.)

The Robin exponential map extends continuously to $\partial K_{\rho}$ via $F(be^{i\theta}) = a+be^{i\theta}+\bar b e^{-i\theta}$.  {  Given $x\in K$ and $v\in\RR^n$, there is always an extremal ellipse through $x$ with tangent in the direction of $v$ (see \cite{burnslevmaurevesz:monge}, section 3). Hence $F:\partial K_{\rho}\to K$ is onto but not injective.}

\begin{theorem}\label{thm:r2}
Let $K\subset\RR^n$ be a convex body with unique extremals.  Then for any $\phi$ continuous on $K$, 
$$
\int \phi(dd^cV_K)^n\ =\ \int(\phi\circ F)(dd^c\rho_K^+)^n.
$$ 
\end{theorem}

We will prove this using Theorem \ref{thm:r1} and an approximation argument.  Let $\{K_j\}$ be a strictly decreasing sequence of compact sets with the following properties:
\begin{enumerate}
\item For each $j$, $K_j$ is the closure of a smoothly bounded, strongly lineally convex domain $D_j$.
\item $K_{j+1}\subset K_j$ with $\bigcap_j K_j = K$.  
\end{enumerate}
For convenience, let us denote the Siciak-Zaharjuta extremal functions by $V_j=V_{K_j}$ and $V=V_K$, and also write  $\rho = \rho_K^+=\max\{\rho_K,0\}$ and  $\rho_j=\rho_{K_j}^+$.  Write $(K_j)_{\rho}=\{\rho_{K_j}\leq 0\}$ and $K_{\rho}=\{\rho_K\leq 0\}$ for the Robin indicatrices, and write $F_j:\CC^n\setminus(K_j)_{\rho}\to\CC^n\setminus K_j$ and $F:\CC^n\setminus K_{\rho}\to\CC^n\setminus K$ for the Robin exponential maps.  

A key ingredient in the approximation will be the following result, stated without proof (cf., \cite{burnslevmau:exterior}, Corollary 7.2).

\begin{proposition}
On any compact subset of $\CC^n\setminus K_{\rho}$ we have the uniform convergence $F_j\to F$. \qed
\end{proposition}

Given $r>1$, we put
\begin{eqnarray*}
V_{j,r}&=&\max\{0,V_j-\log r\}, \quad \rho_{j,r} \ = \ \max\{0,\rho_j-\log r\}, \\
V_r &=& \max\{0,V-\log r\}, \quad \ \; \, \rho_r\ = \ \max\{0,\rho-\log r\}.
\end{eqnarray*}
It is easy to see that $V_{j,r}$ is the extremal function for the set $K_{j,r}=\{V_j\leq \log r\}$ and $\rho_{j,r}$ is the Robin function for $K_{j,r}$, and that $V_{j,r}\nearrow V_r$, $\rho_{j,r}\nearrow\rho_r$ as $j\to\infty$.

\begin{remark}\label{rem:} \rm By Lempert theory, the set $K_{j,r}$ is smoothly bounded and strongly lineally convex.  Also, if $F_{j,r}$ denotes the corresponding Robin exponential map, then   
$
F_{j,r}=F_j \hbox{ on the domain of } F_{j,r}.
$
This follows from the fact that the images of Lempert extremal disks for $V_{j,r}$ are contained in those of  $V_{j}$. Precisely, if $f(\zeta)$ parametrizes a Lempert extremal for $V_j$, then $f_r(\zeta):=  f(r\zeta)$ parametrizes a Lempert extremal disk for $V_{j,r}$. \end{remark}

We will also make use of the following lemma whose proof is straightforward.

\begin{lemma}\label{lem:4.9}
Suppose we have the uniform convergence $\varphi_j\to\varphi$ of continuous functions on a domain $D$ and the weak-$*$ convergence $\mu_j\to\mu$ of measures on $D$, and the total masses $\mu_j(D),\mu(D)$ are uniformly bounded above.  

Then $\int\varphi_j\,d\mu_j\to\int\varphi\, d\mu$. \qed
\end{lemma}
We will apply this to Monge-Amp\`ere measures of functions in $L^+(\CC^n)$, which have total mass $(2\pi)^n$.  

\medskip

We can now prove Theorem \ref{thm:r2}.

\begin{proof}[Proof of Theorem \ref{thm:r2}]
Let $\phi$ be a real-valued continuous function on $K$. Form the continuous function 
$$
\tilde\phi(z) = \left\{ \begin{array}{rl} \phi(z) & \hbox{ if } z\in K\\ \phi(f(\frac{\zeta}{|\zeta|})) 
& \hbox{ if } z=f(\zeta)\in\CC^n\setminus K \end{array}\right.
$$
where $f$ parametrizes an extremal disk that goes through $z$.  Note that $\tilde\phi$ is continuous since the foliation of extremals for $V_K$ is continuous and each leaf extends holomorphically across $K$ (as a complex ellipse).

Fix $r>1$.  By Theorem \ref{thm:r1}, we have
\begin{equation}\label{eqn:}
\int\tilde\phi(dd^cV_{j,r})^n\ = \ \int(\tilde\phi\circ F_{j,r})(dd^c\rho_{j,r})^n
\ = \ \int(\tilde\phi\circ F_j)(dd^c\rho_{j,r})^{n},
\end{equation}
where we use the observation in Remark \ref{rem:} to get the second equality. 

 The Monge-Amp\`ere formula (\ref{eqn:4.5a}) applied to $\rho_{j,r}$, $\rho_r$ shows that $(dd^c\rho_{j,r})^n$, $(dd^c\rho_r)^n$ are supported on the sets $\{\rho_j=\log r\}$, $\{\rho=\log r\}$, which are contained in $\CC^n\setminus K_{\rho}$. Here we have $F_j\to F$ uniformly; hence $\tilde\phi\circ F_j\to\tilde\phi\circ F$ uniformly on a compact set $S\subset\CC^n\setminus K_{\rho}$ that contains a neighborhood of $\{\rho=\log r\}$ and hence contains  $\{\rho_j=\log r\}$ for all sufficiently large $j$.     
  Applying Lemma \ref{lem:4.9}, we have
 $$
 \int(\tilde\phi\circ F_j)(dd^c\rho_{j,r})^{n} \longrightarrow \int(\tilde\phi\circ F)(dd^c\rho_r)^n \quad\hbox{as } j\to\infty.
 $$
The standard Monge-Amp\`ere convergence also gives  $$\int\tilde\phi(dd^cV_{j,r})^n \longrightarrow\int\tilde\phi(dd^cV_r)^n\quad\hbox{as } j\to\infty.$$   Since (\ref{eqn:}) is true for all $j$,  
$\int\tilde\phi(dd^cV_r)^n =  \int(\tilde\phi\circ F)(dd^c\rho_r)^n$ follows by taking the limit as $j\to\infty$.
This latter formula is true for all $r>1$.  

Since $V_r\nearrow V$ and $\rho_r\nearrow \rho$ as $r\to 1^-$, we also have the weak-$*$ convergences $(dd^cV_r)^n\to(dd^cV)^n$ and $(dd^c\rho_r)^n\to(dd^c\rho)^n$.  Taking the limit as $r\to 1^-$ yields $$\int\tilde\phi(dd^cV)^n =  \int(\tilde\phi\circ F)(dd^c\rho)^n.$$  Finally, note that on $K$, where $(dd^cV)^n$ is supported, we have $\tilde\phi=\phi$.  Similarly, it is easy to check that $\tilde\phi\circ F=\phi\circ F$ on the support of $(dd^c\rho)^n$.   The theorem is proved.
\end{proof}

\begin{remark} \rm
 Note that the formula in Theorem \ref{thm:r2} exhibits $(dd^cV_K)^n$ 
as the push-forward of the measure $(dd^c\rho^+_K)^n$ under the Robin exponential map:  $$F_*((dd^c\rho_K^+)^n)=(dd^cV_K)^n.$$
\end{remark}

\bibliographystyle{abbrv}
\bibliography{myreferences}

\end{document}